\documentclass[
a4paper,reqno]{amsart}
\usepackage[T1]{fontenc}
\usepackage[english]{babel}
\usepackage{amssymb,bbm,enumerate}
\usepackage{color}
\usepackage[linktocpage=true,colorlinks=true, linkcolor=blue, citecolor=red, urlcolor=green]{hyperref}
\usepackage[all]{xy}
\usepackage{tikz-cd}

\usepackage{float}
\restylefloat{table}

\usepackage{tikz}

\theoremstyle{plain}
\newtheorem{theorem}{Theorem}[section]
\newtheorem{lemma}[theorem]{Lemma}
\newtheorem{proposition}[theorem]{Proposition}
\newtheorem{corollary}[theorem]{Corollary}

\theoremstyle{definition}
\newtheorem{definition}[theorem]{Definition}
\newtheorem{example}[theorem]{Example}

\theoremstyle{remark}
\newtheorem{remark}[theorem]{Remark}

\numberwithin{equation}{section}

\definecolor{light}{gray}{.9}

\begin{document}

\title[Affine Rota--Baxter groups and affine skew braces]{Affine Rota--Baxter groups and affine skew braces}

\author{Jiayao Ma}
\address{School of Mathematics and Statistics, Northeast Normal University, Changchun, 130024, China}
\email{jiayaoma@nenu.edu.cn}

\author{Boran Zhang}
\address{School of Mathematical Sciences, Peking University, Beijing, 100871, China}
\email{zhangboran@stu.pku.edu.cn}

\author{Jiefeng Liu}
\address{School of Mathematics and Statistics, Northeast Normal University, Changchun, 130024, China}
\email{liujf534@nenu.edu.cn}



\begin{abstract}
	Rota--Baxter groups and skew braces are closely related algebraic structures, both providing set-theoretical solutions to the Yang--Baxter equation. In this paper, we extend these structures to the setting of affine schemes. First, we introduce affine Rota--Baxter groups and, by leveraging the duality between affine groups and Hopf algebras via their coordinate rings, prove the equivalence between affine Rota–Baxter groups and co-Rota–Baxter Hopf algebras. Next, we show affine Rota--Baxter groups can naturally give rise to the affine skew braces defined by Angiono, Galindo, and Vendramin. Conversely, any affine skew brace can be embedded into an affine Rota--Baxter group.
    By linking these to the relationship between affine skew braces and Hopf co-braces, we give new connections between co-Rota--Baxter Hopf algebras and Hopf co-braces. Finally, we propose the study of solutions to the Yang–Baxter equation within the framework of affine schemes, demonstrating that affine skew braces naturally give rise to such solutions.
\end{abstract}


\maketitle

\section{Introduction}
The Yang–Baxter equation (YBE) first appeared in 1967 in Yang's work on the quantum many‑body problem \cite{Yan67}. Independently, Baxter derived the same equation in 1972 while studying the eight‑vertex model \cite{Bax72}. Since then, the YBE has become one of the fundamental equations in mathematical physics.  A major development came when Drinfel'd \cite{Dri92} proposed the systematic study of set‑theoretic solutions, which has since evolved into an active research area. Recall that a set-theoretical solution on a set $X$ is a bijective map $r: X \times X \to X \times X$ satisfying the braid‑type relation
\[
 (r \times \operatorname{id})(\operatorname{id} \times r)(r \times \operatorname{id}) = (\operatorname{id} \times r)(r \times \operatorname{id})(\operatorname{id} \times r).
\]
A central problem in this area is the construction and classification of solutions. The pioneer works on  set-theoretical  solutions were made by Etingof--Schedler--Soloviev \cite{ESS99} and Lu--Yan--Zhu \cite{LYZ00}. 

Rump introduced the notion of braces as a generalization of radical rings, which give rise to involutive set-theoretical solutions \cite{Rump06, Rump07}. This concept was later reformulated by Cedó, Jespers and Okniński \cite{CJO} and subsequently generalized by Guarnieri and Vendramin to skew braces \cite{GV17}, which can be used to construct non‑involutive set‑theoretic solutions. The notion of Hopf braces was introduced by Angiono, Galindo, and Vendramin in \cite{AGV17} as an algebraic framework motivated by the Yang--Baxter equation, encompassing both Rump’s original braces and their non-commutative extensions \cite{GV17} as special cases. Moreover, they developed the concept of affine (skew) braces, thereby lifting brace theory to the category of affine schemes, and further established an equivalence between affine skew braces and commutative Hopf co-braces, where the latter are exactly the finite duals of Hopf braces. 

 Guo, Lang and Sheng \cite{GLS21} first introduced the notion of Rota--Baxter (Lie) groups in their study of the integration of Rota--Baxter Lie algebras. This notion was later further explored by Bardakov and Gubarev \cite{BG22}, who revealed a deep connection between Rota–Baxter groups and skew braces. In particular, it has been shown that every Rota--Baxter group gives rise to a skew brace, and conversely, very skew brace can be embedded into a Rota--Baxter group. For more details and references on Rota--Baxter groups, we refer the reader to see \cite{BG23, DR23, GG23, J25}. In \cite{Gon}, Goncharov introduces the notion of a Rota--Baxter operator on a cocommutative Hopf algebra; a cocommutative Hopf algebra endowed with such an operator is called a Rota--Baxter Hopf algebra. Moreover, he proves that every Rota--Baxter group admits a unique extension to a Rota--Baxter Hopf algebra on its group algebra, and that Rota--Baxter Lie algebras are in bijective correspondence with Rota--Baxter Hopf algebra structures on their universal enveloping algebras. In \cite{ZZL24}, Zheng, Zhao and Liu introduce the notion of  a Rota--Baxter co-operator on a commutative Hopf algebra as a dual of Rota--Baxter operator on a cocommutative Hopf algebra; a commutative Hopf algebra endowed with such an operator is called a co-Rota--Baxter Hopf algebra. The relationships between Hopf co-braces and co-Rota--Baxter Hopf algebra are also analyzed.  

Inspired by the known connections between Rota--Baxter groups and skew braces, this paper introduces affine (algebraic) Rota--Baxter groups, lifting both structures and their subtle interactions from the set-theoretic setting to the framework of affine groups. 

We first establish a categorical equivalence between affine Rota--Baxter groups and co-Rota--Baxter Hopf algebras. We then prove that every affine Rota--Baxter group naturally gives rise to an affine skew brace, and conversely, that every affine skew brace embeds into an affine Rota--Baxter group. Utilizing the established equivalence between affine skew braces and commutative Hopf co-braces, we provide a conceptual, computation-free proof of the connection between co-Rota--Baxter Hopf algebras and Hopf co-braces.

Finally, we generalize set-theoretic solutions of the Yang--Baxter equation to affine schemes, expressing solutions as natural transformations instead of single maps. Crucially, we show that every affine skew brace naturally yields such a solution. The central constructions and results are summarized in the diagram below.
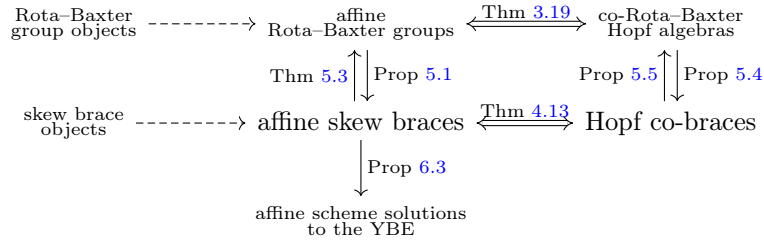
\begin{figure}[htbp]
\centering
\begin{tikzcd}[column sep= large]
 \substack{\text{Rota--Baxter} \\ \text{group objects}} \ar[r,dashed] & \substack{\text{affine} \\ \text{Rota--Baxter groups}} \ar[r, "\text{Thm \ref{coRBHopf = AffRBgroup category equivalence}}", Leftrightarrow] \ar[d,shift left, "\text{Prop \ref{affine RB to affine skew brace prop}}"] & \substack{\text{co-Rota--Baxter} \\ 
 \text{Hopf algebras}} \ar[d, shift left, "\text{Prop \ref{coRBHopf give rise to Hopf cobrace}}"] \\
 \substack{\text{skew brace} \\ \text{objects}} \ar[r,dashed] & \text{affine skew braces} \ar[d, "\text{Prop \ref{prop:affine-skew-brace-to-solution}}"] \ar[r, "\text{Thm \ref{affine skewbrace Hopf co-brace}}", Leftrightarrow] \ar[u,shift left, "\text{Thm \ref{thm:affine-skew-brace-embedding}}"] & \text{Hopf co-braces} \ar[u,shift left, "\text{Prop \ref{new relation from Hopfcobrace to coRB}}"] \\
  &\substack{\text{affine scheme solutions} \\ \text{to the YBE}}  &
\end{tikzcd}
\end{figure}

The paper is organized as follows.  In Section \ref{sec:pre}, we first recall the notions of Rota--Baxter groups,  (skew) braces, and  their relations. Then we recall the concepts of affine groups and Hopf algebra structures on their corresponding coordinate rings. In Section \ref{sec:affine}, we introduce the notion of an affine Rota--Baxter group and present several explicit constructions. We then establish an equivalence between affine Rota--Baxter groups and co-Rota--Baxter Hopf algebras. In Section \ref{sec:brace}, we first recall the notion of an affine (skew) brace and give three equivalent descriptions. Then we recall the equivalence between affine skew braces and commutative Hopf co-braces. In Section \ref{sec:correspond}, we explore the relationship between affine Rota--Baxter  groups and affine skew braces. Specifically, we prove that every affine Rota--Baxter  group gives rise to an affine skew brace, and conversely, that every affine skew brace can be embedded into some affine Rota--Baxter  group. We then prove new connections between co-Rota--Baxter Hopf algebras and Hopf co-braces. In Section \ref{sec:YB}, we introduce the Yang--Baxter equation in the setting of affine schemes, and prove that every affine skew brace gives such a solution in a natural way.

Throughout this paper, for simplicity, we assume that $k$  is a field; however, many of the results remain valid if $k$
is only assumed to be a commutative ring. All \(k\)-algebras are assumed to be commutative and unital, and all algebra homomorphisms are unital.

\section{Preliminaries}\label{sec:pre}
\subsection{Rota--Baxter groups and (skew) braces}
We begin by recalling the definitions of Rota--Baxter groups, and (skew) braces, as well as their relations.

\begin{definition}{\rm (\cite{GLS21})}
A map \(B: G \to G\) on a group \(G\) is called a \textbf{Rota--Baxter operator} if
\begin{equation}
\label{eq:RB-operator}
    B(g)B(h) = B(gB(g)hB(g)^{-1}) \qquad \forall g,h \in G.
\end{equation}
A \textbf{Rota--Baxter group} is a pair \((G, B)\) consisting of a group \(G\) and a Rota--Baxter operator \(B: G \to G\). We will often refer to a Rota--Baxter operator simply as an \textbf{RB-operator}, and a Rota--Baxter group simply as an \textbf{RB group}.
\end{definition}

A {\bf Rota--Baxter group homomorphism} from $(G,B)$ to $(G', B')$ is a group homomorphism $\varphi$ satisfying $\varphi \circ B = B' \circ \varphi$. Moreover, RB groups together with their homomorphisms form the category of RB groups, denoted $\mathsf{RBGrp}$.

\begin{example}{\rm (\cite{GLS21})}
\label{prelim examples}
Let \( G \) be a group. Then
\begin{enumerate}
    \item the map \( B_0(g) = e \) is an RB-operator on \( G \);
    \item the map \( B_{-1}(g) = g^{-1} \) is an RB-operator on \( G \);
    \item given an exact factorization \( G = HL \), the map \( B : G \to G \) defined by \( B(h\ell) = \ell^{-1} \), where $h\in H$ and $\ell \in L$, is an RB operator on \( G \).
\end{enumerate}
\end{example}

\begin{proposition}{\rm (\cite{GLS21})}
\label{G oB group}
    Let $(G,B)$ be an {\rm RB} group.
    \begin{enumerate}
    \item[{\rm (1)}] \label{Item(1) of Prop2.3} The pair $(G, \circ_B)$, with the multiplication
    \begin{equation}
        g \circ_B h = g B(g) h B(g)^{-1}, \quad g, h \in G
    \end{equation}
   is also a group. The group $G_B :=(G,\circ_B)$ is called the {\bf descendant group} of the {\rm RB} group $(G,B)$.

    \item[{\rm (2)}] \label{Item(2) of Prop2.3} The operator $B$ is an {\rm RB} operator on the group $(G, \circ_B)$.

    \item[{\rm (3)}] \label{Item(3) of Prop2.3} The map $B: (G, \circ_B) \rightarrow (G, \cdot)$ is a homomorphism of {\rm RB}  groups.
\end{enumerate}
\end{proposition}

\begin{remark}
\label{o^B remark}
Let \(e\) be the identity of the RB group \((G,B)\). One checks that \(B(e) = e\), and that \(e\) is also the identity with respect to  the new operation \(\circ_B\). The inverse of \(g\) with respect to \(\circ_B\) is $B(g)^{-1} g^{-1} B(g)$.
\end{remark}

Then we recall the notion of a (skew) brace, which is closely related to RB groups.

\begin{definition}{\rm (\cite{CJO,GV17})}
A \textbf{skew left brace} is a triple \((A, \cdot, \circ)\) where \((A, \cdot)\) and \((A, \circ)\) are groups such that
\begin{equation}
a \circ (b \cdot c) = (a \circ b) \cdot a^{-1} \cdot (a \circ c) \qquad \forall a,b,c \in A,
\end{equation}
where \(a^{-1}\) denotes the inverse of \(a\) in \((A,\cdot)\). If in addition \((A,\cdot)\) is abelian, then \((A,\cdot,\circ)\) is called a \textbf{left brace}.
\end{definition}

The notion of (skew) right braces can be defined analogously, and there exists a one-to-one correspondence between (skew) left braces and (skew) right braces. Henceforth, we will simply refer to (skew) left braces as (skew) braces unless otherwise stated.

A {\bf homomorphism of (skew) braces} from \((A,\cdot_A, \circ_A) \) to \((A',\cdot_{A'}, \circ_{A'}) \) is a map
\( \psi : A \to A' \) such that, for all \( a, b \in A \), $\psi(a \cdot_A b) = \psi(a) \cdot_{A'} \psi(b)$ and $\psi (a \circ_A b) = \psi(a) \circ_{A'} \psi(b)$. Thus, the braces (resp. skew braces) and their homomorphisms form the category of braces \(\mathsf{Br}\) (resp. skew braces \(\mathsf{SBr}\)).

\begin{lemma}{\rm (\cite{GV17})}
	\label{skew brace lemma}
	Let \((A, \cdot, \circ)\) be a skew brace. Then \(e = e_0\), where \(e\) denotes the identity of \((A, \cdot)\) and \(e_0\) denotes the identity of \((A, \circ)\).
\end{lemma}

\begin{proposition}{\rm (\cite{GV17})}\label{lamda from skewbrace remark}
	Let \((G, \cdot, \circ)\) be a skew  brace. For each \(a \in G\), define a map
	\[
	\lambda_a: G \to G, \quad \lambda_a(b) = a^{-1} \cdot (a \circ b).
	\]
Then we have
	\begin{enumerate}
		 \item[{\rm (1)}]  For each \(a \in G\), \(\lambda_a\) is an automorphism of \((G, \cdot)\).
		 \item[{\rm (2)}]  The map \(\lambda: (G, \circ) \to \operatorname{Aut}(G, \cdot)\), \(a \mapsto \lambda_a\), is a group homomorphism.
		 \item[{\rm (3)}]  The inverse of \(a\) with respect to \(\circ\) is given by \(a^{\circ(-1)} = \lambda_a^{-1}(a^{-1})\).
	\end{enumerate}
\end{proposition}

\begin{example}{\rm (\cite{GV17})}
(1) Let \((G, \cdot)\) be a group. Then \((G, \cdot, \cdot)\) is a (trivial) skew brace.
	
(2) Let \(A\) and \(B\) be groups and let \(\alpha: A \to \operatorname{Aut}(B)\) be a group homomorphism. Then the direct product \(A \times B\) becomes a skew brace with the operations
	\[
	(a, b) \cdot (a', b') = (aa', bb'), \qquad
	(a, b) \circ (a', b') = (aa', \, b \alpha_a(b')),
	\]
	for all \(a, a' \in A\) and \(b, b' \in B\).
\end{example}

We now recall the relationship between RB groups and skew braces.

\begin{proposition}{\rm (\cite{BG22})}
\label{RBgroup to skew brace  proposition}
Let \((G, \cdot,B)\) be an {\rm RB} group. Put
\begin{equation}
    a \circ_B b = a \cdot B(a) \cdot b \cdot B(a)^{-1}.
\end{equation}
Then \((G, \cdot, \circ_B)\) is a skew brace.
\end{proposition}

Conversely, let \((G, \cdot, \circ)\) be a skew brace. Consider the semi-direct product group \(\tilde{G} = (G, \circ) \ltimes (G, \cdot)\) with the operation 
$$(a, b) * (c, d) = (a \circ c, b\lambda_a(d)).$$
Then \(H = \{(g, g) \mid g \in G\}\) and \(L = \{(g, e) \mid g \in G\}\) are subgroups of $\tilde{G}$. Since
$(a, b) = (b, b) * (b^{\circ(-1)} \circ a, e)$, we have an exact factorization \(\tilde{G} = H * L\).
Define an RB-operator \(B\) on \(g = h * l \in \tilde{G}\) as \(B(h * l) = l^{*(-1)}\), namely, $B((a, b)) = (a^{\circ(-1)} \circ b, e).$ Define a map $\psi:G\rightarrow \tilde{G}$ by $\psi(g)=(e,g)$.  One can check that \(\psi\) is an isomorphism of skew  braces \(G\) and \(\operatorname{Im}(\psi)\), where \(\operatorname{Im}(\psi)\) is considered as a skew subbrace of \(\tilde{G}(B) = (\tilde{G}, *, \circ_B)\). Thus we obtain

\begin{theorem}{\rm (\cite{BG22})}
\label{embed skew brace into RBG Thm}
Every skew brace can be embedded into an {\rm RB} group.
\end{theorem}

\subsection{Group objects, affine groups and Hopf algebras}
Let \(\mathsf{Set}\) be the category of sets and \(\mathsf{Alg}_k\) the category of commutative \(k\)-algebras.  
For a category \(\mathcal{C}\), we write \(\mathsf{Func}[\mathcal{C}, \mathsf{Set}]\) for the category of functors from \(\mathcal{C}\) to sets, and \(\mathsf{RepFunc}[\mathcal{C}, \mathsf{Set}]\) for its full subcategory of representable functors. 

Throughout this paper, we always assume that \(\mathcal{C}\) is locally small and has finite products. In particular, \(\mathcal{C}\) has a terminal object \(*\) with canonical isomorphisms \(S \times * \cong S \cong * \times S\) for any object \(S\).

For a category with finite products and coproducts, we denote the product of two objects \(A_1, A_2\) by \(A_1 \times A_2\), with projections \(\pi_i: A_1 \times A_2 \to A_i\).  
Their coproduct is denoted by \(A_1 \otimes A_2\), with injections \(\iota_i: A_i \to A_1 \otimes A_2\).  
The universal properties of products and coproducts give the morphism \(\langle f_1, f_2 \rangle: X \to A_1 \times A_2\) for \(f_i: X \to A_i\), and the morphism \([f_1, f_2]: A_1 \otimes A_2 \to X\) for \(f_i: A_i \to X\). 

Moreover, for morphisms \(f_i: A_i \to B_i\) (\(i=1,2\)), we set
\[
f_1 \times f_2 := \langle f_1 \circ \pi_1, \, f_2 \circ \pi_2 \rangle: A_1 \times A_2 \longrightarrow B_1 \times B_2,
\]
\[
f_1 \otimes f_2 := [\iota_1 \circ f_1, \, \iota_2 \circ f_2]: A_1 \otimes A_2 \longrightarrow B_1 \otimes B_2.
\]

Write $A^{\times n}$ for the $n$-fold product and $A^{\otimes n}$ for the $n$-fold coproduct. 
For $\varphi: A \times A \to A$, define $\varphi^{(1)} = \varphi$ and 
\begin{equation}
\label{phi(n):A(n+1) to A}
    \varphi^{(n+1)} = \varphi \circ (\varphi^{(n)} \times \mathrm{id}) : A^{\times (n+2)} \to A.
\end{equation}
For $\psi:A \otimes A \to A$, define $\psi^{(1)} = \psi$ and 
\begin{equation}
\label{psi(n):A(n+1) to A}
    \psi^{(n+1)} =\psi \circ (\psi^{(n)}\otimes \mathrm{id}):A ^{\otimes (n+2)} \to A.
\end{equation}
For $\psi: A \to A \otimes A$, define $\psi^{(1)} = \psi$ and 
\begin{equation}
\label{psi(n):A to A(n+1)}
    \psi^{(n+1)} = (\psi^{(n)} \otimes \mathrm{id}) \circ \psi : A \to A ^{\otimes (n+2)}.
\end{equation}
These notations extend to other objects and morphisms when context permits.

Consider an object $A$ of $\mathcal{C}$. Then $A$ defines a functor $h^A:\mathcal{C} \rightarrow \mathsf{Set}$ by
\[\left\{
\begin{aligned}
&R \mapsto h^A(R):=\mathrm{Hom}(A, R),\quad R \in \mathrm{ob}(\mathcal{C})\\
&f \mapsto h^A(f),\mbox{ with } h^A(f)(g):=f \circ g, \quad f: R \to R', ~g \in h^A(R).
\end{aligned}
\right.\]
A morphism $\alpha: A' \to A$ defines a map $f \mapsto f \circ \alpha: h^{A}(R) \to h^{A'}(R)$ which is natural in $R$. Thus $A \rightsquigarrow h^{A}$ is a contravariant functor $\mathcal{C} \to \mathsf{Func}[\mathcal{C}, \mathsf{Set}]$. Symbolically, $h^A = \operatorname{Hom}(A, -)$ is covariant. Similarly, define the contravariant functor $h_A=\operatorname{Hom}(-,A): \mathcal{C} \to \mathsf{Set}$. Then $A \rightsquigarrow h_{A}$ is a covariant functor $\mathcal{C} \to \mathsf{Func}[\mathcal{C}^{\mathrm{opp}}, \mathsf{Set}]$.

The {\bf Yoneda lemma} states that for any functor \(F: \mathcal{C} \to \mathsf{Set}\) and any object \(A \in \mathcal{C}\), there is a bijection $\operatorname{Nat}(h^A, F) \cong F(A)$, which is natural in both \(A\) and \(F\). In particular, taking \(F = h^B\) gives $\operatorname{Nat}(h^A, h^B) \cong \operatorname{Hom}(B, A)$.
Thus, the functor \(A \rightsquigarrow h^A\) is fully faithful. 

A functor $F:\mathcal{C}\rightarrow\mathsf{Set}$ is said to be {\bf representable} if it is isomorphic to $h^A$ for some object $A$ in $\mathcal{C}$. The Yoneda lemma says that $A \rightsquigarrow h^A$ is a contravariant equivalence from $C$ onto the category of representable functors \(\mathsf{RepFunc}[\mathcal{C}, \mathsf{Set}]\).

Recall that a \textbf{group object} in a category \(\mathcal{C}\) is an object \(G \in \mathcal{C}\) together with a morphism \(m: G \times G \to G\) for which there exist morphisms \(e: * \to G\) and \(\operatorname{inv}: G \to G\) satisfying the associativity, unit, and inverse laws (expressed by commutative diagrams).  We denote a group object by \((G, m)\) when the unit and inverse are understood, or by \((G, m, e, \operatorname{inv})\) when we need to make them explicit.

It is clear that  a group object in $\mathsf{Set}$ is just a group. Note that the tensor product is the coproduct in the category \(\mathsf{Alg}_k\), so the category \(\mathsf{RepFunc}[\mathsf{Alg}_k, \mathsf{Set}]\) has finite products.

Then we recall the notion of affine groups and several equivalent characterizations that will be useful. For a comprehensive introduction to the theory of affine groups (scheme) we refer to \cite{Mil12} and \cite{W1979}.

\begin{remark}
A representable functor \(\mathcal{A}: \mathsf{Alg}_k \to \mathsf{Set}\) is called an affine scheme.
\end{remark}

\begin{definition}
\label{def affine group}
An \textbf{affine group scheme over $k$} (or simply an \textbf{affine group}) is a representable functor $G \colon \mathsf{Alg}_k \to \mathsf{Set}$ together with a natural transformation $m \colon G \times G \to G$ such that, for any $R \in \mathsf{Alg}_k$, $m_R \colon G(R) \times G(R) \to G(R)$ is a group structure on $G(R)$. If $G$ is represented by a finitely generated $k$-algebra, then it is called an \textbf{affine algebraic group}.
\end{definition}

An \textbf{affine (algebraic) group homomorphism} from $G$ to $H$ is a family of homomorphisms $\Phi_R : G(R) \to H(R)$
of groups such that $\Phi_{R'} \circ G(f)=H(f) \circ \Phi_R$ for  \(f \in \operatorname{Hom_{\mathsf{Alg}_k}}(R,R')\). In other words, a homomorphism of affine groups is a natural transformation preserving the group structures. 

Thus, affine groups over \(k\) with their homomorphisms form a category, denoted by \(\mathsf{AffGrp}_k\).

\begin{definition}
Let \((G, m)\) be an affine group over \(k\). The {\bf coordinate ring} of \(G\) is the commutative \(k\)-algebra \(A\) such that \(G \cong h^A\). It is denoted by \(\mathcal{O}(G)\).
\end{definition}

The following proposition gives the equivalent definitions of affine groups. 
\begin{proposition}\label{pro:equivalent affine groups}
     The following are equivalent:
    \begin{enumerate}
        \item[{\rm (1)}] $(G,m)$ is an affine group over $k$.
        \item[{\rm (2)}]$(G,m)$ is a group object in $\mathsf{RepFunc}[\mathsf{Alg}_k, \mathsf{Set}]$.
       \item[{\rm (3)}] A functor  $G \colon \mathsf{Alg}_k \to \mathsf{Grp}$ whose underlying set-valued functor is representable, where \(\mathsf{Grp}\) is the the category of groups.
    \end{enumerate}
\end{proposition}

Let $(G, m, e, \operatorname{inv})$ be an affine group, i.e., a group object in $\mathsf{RepFunc}[\mathsf{Alg}_k, \mathsf{Set}]$. For each $R\in\mathsf{Alg}_k$, the set $G(R)$ is a group with multiplication $m_R$, unit $e_R$, and inversion $\operatorname{inv}_R$.

\begin{remark}\label{Hom on affine group}
Since $G$ is representable, $G\cong h^{\mathcal{O}(G)}$, where $\mathcal{O}(G)$ is the coordinate ring. By the Yoneda lemma, the natural transformations $m$, $e$, and $\operatorname{inv}$ correspond uniquely to $k$-algebra homomorphisms
\begin{align*}
\Delta &: \mathcal{O}(G) \to \mathcal{O}(G)\otimes\mathcal{O}(G),\\
\varepsilon &: \mathcal{O}(G) \to k,\\
S &: \mathcal{O}(G) \to \mathcal{O}(G),
\end{align*}
satisfying, for any $f,g\in G(R)\cong \operatorname{Hom}(\mathcal{O}(G),R)$,
\begin{enumerate}
\item[\rm(1)] $f\cdot g = m_R(f,g) = \mu_R \circ (f\otimes g)\circ \Delta$, where $\mu_R:R \otimes R \to R$ is the multiplication map of $R$, 
\item[\rm(2)] the identity element is $\tilde{e}= e_R(\eta_R) = \eta_R\circ\varepsilon$, where $\eta_R:k\to R$ is the unit map of $R$,
\item[\rm(3)] $f^{-1} = \operatorname{inv}_R(f) = f\circ S$.
\end{enumerate}
\end{remark}

Its coordinate ring \(\mathcal{O}(G)\) then becomes a commutative Hopf algebra with comultiplication \(\Delta\), counit \(\varepsilon\) and antipode \(S\) corresponding to \(m, e, \operatorname{inv}\) respectively under the Yoneda lemma.

\begin{proposition}
\label{Prop affinegroup equivalent  to Hopf}
The above correspondence  gives an equivalence of categories:
\[
\mathsf{AffGrp}_k \simeq \mathsf{Hopf}_k^{\mathrm{opp}},
\]
where \(\mathsf{Hopf}_k^{\mathrm{opp}}\) denotes the opposite category of commutative Hopf algebras over \(k\).
\end{proposition}

\section{Affine Rota--Baxter groups and co-Rota--Baxter Hopf algebras}\label{sec:affine}
In this section, we first introduce the notion of an affine Rota--Baxter group. This construction naturally induces a new affine group, which we call a descendant affine group. We then establish an equivalence between affine Rota--Baxter groups and co-Rota--Baxter Hopf algebras. Finally, we study the relationships between descendant affine groups and  descendant Hopf algebras. 
\subsection{Rota--Baxter group objects}
\label{sec:RBG in C}
We first extend the notion of RB groups to a locally small category $\mathcal{C}$ equipped with finite products.

\begin{definition}\label{Rota--Baxter group object def}
    A \textbf{Rota--Baxter group object in \(\mathcal{C}\)} is a quintuple \((G, m, e, \operatorname{inv}, B)\), where \((G, m, e, \operatorname{inv})\) is a group object in \(\mathcal{C}\) and \(B: G \to G\) is a morphism satisfying the following commutative diagram:
\begin{equation}
\label{Rota--Baxter operator in C defined by diagram}
    \begin{tikzcd}[column sep = 2.5cm, row sep = huge]
     G^{\times 2} \ar[r,"B\times B"] \ar[d,"{\langle \pi_1, \pi_1, \pi_2, \pi_1\rangle}"]  & G^{\times 2} \ar[r,"m"]  & G \\G^{\times 4} \ar[r,"\mathrm{id} \times B \times \mathrm{id} \times (\mathrm{inv} \circ B)"]  & G^{\times 4} \ar[r,"m^{(3)}"] & G \ar[u,"B"],
    \end{tikzcd}
\end{equation}
where $m^{(3)}$ is given by \eqref{phi(n):A(n+1) to A}.
For brevity, we often denote it as the triple \((G, m, B)\). We call $B:G \rightarrow G$ a \textbf{Rota--Baxter operator on group object $(G,m)$} in $\mathcal{C}$.
\end{definition}

\begin{remark}
\label{remark RBG object in Set}
When we consider $\mathcal{C}=\mathsf{Set}$, then $G$ is a set and Diagram \ref{Rota--Baxter operator in C defined by diagram} implies:  for any $g , h \in G$,
\[
B(g)B(h)=B\big(gB(g)hB(g)^{-1} \big).
\]
Thus a Rota--Baxter group object in $\mathsf{Set}$ is precisely a Rota--Baxter group.
\end{remark}

\begin{example}
A Rota--Baxter group object $(G,m,e,\operatorname{inv},B)$ in the category of smooth manifolds is a Rota--Baxter Lie group. Indeed, $m$, $e$, $\operatorname{inv}$ and $B$ are all required to be smooth, which is exactly the definition of a Rota--Baxter Lie group.
\end{example}

\subsection{Affine (algebraic) Rota--Baxter groups}

\begin{definition}
\label{def affine RB group}
    An \textbf{affine Rota--Baxter group scheme over $k$} (or simply an \textbf{affine Rota--Baxter group}) is a representable functor $G\colon \mathsf{Alg}_k \to \mathsf{Set}$ together with two natural transformations $m\colon G\times G \to G$ and $B:G \to G$ such that, for all $R$ in $\mathsf{Alg}_k$, 
    \[
        m_R\colon G(R)\times G(R) \to G(R) 
    \]
    and
    \[
        B_R\colon G(R) \to G(R)
    \]
   make $\big(G(R),m_R,B_R \big)$ into a Rota--Baxter group. If $G$ is represented by a finitely generated $k$-algebra, then it is called an \textbf{affine algebraic Rota--Baxter group}. We call $B$ a \textbf{Rota--Baxter operator on the affine (algebraic) group $(G,m)$}.
\end{definition}

In what follows, for an affine Rota--Baxter group we usually write \((G,B)\) or \((G, m, B)\), and we will often refer to it simply as an \textbf{affine RB group}.

\begin{definition}
\label{def affRBG homomorphism}
Let \((G, m_G, B_G)\) and \((H, m_H, B_H)\) be affine (algebraic) Rota--Baxter groups over \(k\). A \textbf{homomorphism} \(\Phi: (G, m_G, B_G) \to (H, m_H, B_H)\) is a natural transformation such that for every \(k\)-algebra \(R\), the map
\[
\Phi_R: G(R) \longrightarrow H(R)
\]
is a homomorphism of Rota--Baxter groups. Equivalently, the following diagrams commute:
\begin{equation}
\begin{tikzcd}
G \times G \arrow[r, "m_G"] \arrow[d, "\Phi \times \Phi"'] & G \arrow[d, "\Phi"] \\
H \times H \arrow[r, "m_H"'] & H
\end{tikzcd}
\qquad
\begin{tikzcd}
G \arrow[r, "B_G"] \arrow[d, "\Phi"'] & G \arrow[d, "\Phi"] \\
H \arrow[r, "B_H"'] & H .
\end{tikzcd}
\end{equation}
\end{definition}

Therefore, affine Rota--Baxter groups over $k$ and the morphisms between them form a category, called the category of affine Rota--Baxter groups and denoted by $\mathsf{AffRBGrp}_k$.

\begin{definition}
The {\bf coordinate ring} of an affine Rota--Baxter group \((G, m, B)\) is defined as the coordinate ring of its underlying affine group, denoted by \(\mathcal{O}(G)\) or \(\mathcal{O}(G, m, B)\).
\end{definition}

\begin{example}
Define the functor \(\mathbb{G}: \mathbf{Alg}_k \to \mathbf{Set}\) by \(\mathbb{G}(R) = R^\times\), where \(R^\times\) denotes the set of multiplicative units in \(R\), and \(\mathbb{G}(f) = f\) for any homomorphism \(f: R \to S\) of commutative \(k\)-algebras. Let \(m: \mathbb{G} \times \mathbb{G} \to \mathbb{G}\) be the natural transformation given by \(m_R(x, y) = x \cdot_R y\), where \(\cdot_R\) is the multiplication of \(R\). Therefore \((\mathbb{G}, m)\) is an affine algebraic group represented by \(k[X, X^{-1}]\). Define \(B_R: \mathbb{G}(R) \to \mathbb{G}(R)\) by \(B_R(x) = x^{-1}\), which is natural in \(R\). Then \((\mathbb{G}, m, B)\) is an affine Rota--Baxter group scheme.
\end{example}

\begin{example}
The functor \({\rm GL}_n\) defined by \({\rm GL}_n(R) = \{ A \in M_n(R) \mid \det(A) \in R^\times \}\) is an affine algebraic group, represented by \(k[X_{ij}, \det^{-1}]\). Define \(B_R(A) = A^{-1}\). One can check that $B$ is a natural transformation. Then \((\mathrm{GL}_n, B)\) is an affine algebraic Rota--Baxter group.
\end{example}

Analogous to Proposition~\ref{pro:equivalent affine groups}, we have the following characterizations of affine Rota--Baxter groups.

\begin{theorem}
\label{affineRBgroup THM}
    The following are equivalent:
    \begin{enumerate}
        \item[\rm(1)] $(G,m,B)$ is an affine Rota--Baxter group over $k$.
          \item[\rm(2)]  $(G,m,B)$ is a Rota--Baxter group object in $\mathsf{RepFunc}[\mathsf{Alg}_k, \mathsf{Set}]$.
          \item[\rm(3)]  A functor  $G \colon \mathsf{Alg}_k \to \mathsf{RBGrp}$ whose underlying set-valued functor is representable.
    \end{enumerate}
\end{theorem}
\begin{proof}
First we show the equivalence between (1) and (2). By Proposition~\ref{pro:equivalent affine groups}, an affine group $(G,m)$ is equivalent to a group object $(G,m,e,\operatorname{inv})$ in $\mathsf{RepFunc}[\mathsf{Alg}_k, \mathsf{Set}]$. For $(G,m,B)$ to be an affine RB group, the following diagram should commute for every $R\in\mathsf{Alg}_k$:
\[
    \begin{tikzcd}[column sep = 3cm, row sep = huge]
    G(R)^{\times 2} \ar[r,"B_R\times B_R"] \ar[d,"{\langle \pi_1, \pi_1, \pi_2, \pi_1\rangle}"]  & G(R)^{\times 2} \ar[r,"m_R"]  & G(R) \\G(R)^{\times 4} \ar[r,"\mathrm{id} \times B_R \times \mathrm{id} \times (\mathrm{inv_R} \circ B_R)"]  & G(R)^{\times 4} \ar[r,"m_R^{(3)}"] & G(R) \ar[u,"B_R"] ,
    \end{tikzcd}
\]
where $m_R^{(3)}$ is given by \eqref{phi(n):A(n+1) to A}. Note that this is equivalent to $(G,m,e,\mathrm{inv})$ is a group object in $\mathsf{RepFunc}[\mathsf{Alg}_k, \mathsf{Set}]$ together with a natural transformation $B$ satisfying Diagram \eqref{Rota--Baxter operator in C defined by diagram}, i.e., $(G,m,e,\mathrm{inv},B)$ is an RB group object in $\mathsf{RepFunc}[\mathsf{Alg}_k, \mathsf{Set}]$. 

Next we show (1) is equivalent to (3). Let $(G,m,B)$ be an affine RB group. Because $m$ and $B$ are natural transformations, for any algebra homomorphism $\varphi:R\to R'$, the induced map $G(\varphi): G(R) \to G(R')$ automatically preserves the multiplication $m_R$ and the RB operator $B_R$; hence $G(\varphi)$ is a homomorphism of RB groups. This defines a functor $\mathsf{Alg}_k \to \mathsf{RBGrp}$. 

Conversely, let $G$ be a functor $\mathsf{Alg}_k \to \mathsf{RBGrp}$ whose underlying set-valued functor is representable. For any $R\in \mathsf{Alg}_k$, $G(R)$ is an RB group and we denote the group multiplication by $m_R$ and the RB operator by $B_R$. It remains to check that $m_R$ and $B_R$ are natural in $R$. For any algebra homomorphism $\varphi:R\to R'$, $G(\varphi):G(R)\to G(R')$ is an RB group homomorphism. So we have
\[
G(\varphi) \circ m_R = m_{R'}\circ \big(G(\varphi) \times G(\varphi)\big)
\]
and
\[
G(\varphi)\circ B_R = B_{R'}\circ G(\varphi).
\]
These equalities precisely state that $m$ and $B$ are natural transformations. Hence the representable functor $G$, together with $m$ and $B$, forms an affine RB group.
\end{proof}

\begin{remark}
\label{Hom RBGroup}
The group structure of $G(R) \cong h^{\mathcal{O}(G)}(R)$ is clear in Remark \ref{Hom on affine group}. By Yoneda lemma, there exists a unique algebra homomorphism $\mathcal{B}:\mathcal{O}(G) \to \mathcal{O}(G)$ correponding to $B$  such that 
\[
    B_R(f)=f\circ \mathcal{B}, \qquad \forall f \in h^{\mathcal{O}(G)}(R).
\]
\end{remark}

\begin{proposition}
\label{Be=e affine RBGrp}
    Let $(G,m,e,\operatorname{inv},B)$ be an affine RB group. Then we have $B\circ e =e$.
\end{proposition}
\begin{proof}
    For the RB group $\big(h^{\mathcal{O}(G)}(R), \cdot , B_R \big)$, the RB operator $B_R$ preserve the identity element $\tilde{e}$ of $G(R)$. Therefore, we have $B_R(\tilde{e}) =\tilde{e}$.
    Furthermore, we have $\tilde{e}=\eta_R \circ \varepsilon$ and $B_R(\tilde{e})=\eta_R \circ \varepsilon \circ \mathcal{B}$. Thus, we have
    \begin{align*}
        \eta_R \circ \varepsilon \circ \mathcal{B} =\eta_R \circ \varepsilon, \quad (B\circ e)_R(\eta_R) = e_R(\eta_R),
    \end{align*}
which implies that $B\circ e =e$.
\end{proof}

We now show that an affine Rota--Baxter group admits a new affine group structure, which we call the \textbf{descendant affine group}.

\begin{proposition}
\label{descend affine RBG}
Let $(G, m, e, \operatorname{inv},B)$ be an affine Rota--Baxter group. Then we have:
\begin{enumerate}
    \item[{\rm (i)}] The quadruple $(G,m^B,e,\operatorname{inv}^B)$ is an affine group, called the \textbf{descendant affine group} of the affine Rota--Baxter group $(G, m,e,\operatorname{inv}, B)$. Here the natural transformations $m^B: G \times G \to G$ and $\operatorname{inv}^B:G \to G$ are defined by the following commutative diagrams respectively: 
    \begin{equation}
    \label{m^B diagram}
     \begin{tikzcd}[column sep = 4cm, row sep = huge]
     G^{\times 2} \ar[r, "m^B"] \ar[d,"{\langle \pi_1,\pi_1,\pi_2,\pi_1\rangle}"] 
     & G 
     \\
     G^{\times 4} \ar[r, "{\operatorname{id} \times B \times \operatorname{id} \times (\operatorname{inv} \circ B)}"]
     & G^{\times 4} \ar[u, "m^{(3)}"]
     \end{tikzcd}
    \end{equation}
and
\begin{equation}
    \label{inv^B diagram}
     \begin{tikzcd}[column sep = 4cm, row sep = huge]
     G \ar[r, "\operatorname{inv}^B"] \ar[d,"{\langle \operatorname{id}, \operatorname{id}, \operatorname{id} \rangle}"] 
     & G 
     \\
     G^{\times 3} \ar[r, "(\operatorname{inv} \circ B) \times \operatorname{inv} \times B"]
     & G^{\times 3} \ar[u, "m^{(2)}"] ,
     \end{tikzcd}
    \end{equation}
    where $m^{(2)}$ and $m^{(3)}$ are given by \eqref{phi(n):A(n+1) to A}.
    \item[{\rm (ii)}] The quintuple $(G,m^B,e,\operatorname{inv}^B,B)$ is also an affine Rota--Baxter group.
    \item[{\rm (iii)}] The natural transformation $B$ is a homomorphism of affine Rota--Baxter groups from $(G,m^B,B)$ to $(G,m,B)$.
\end{enumerate}
\end{proposition}
\begin{proof}
$\mathrm{(i)}$ It is clear that $m^B$, $e$ and $\mathrm{inv}^B$ are natural transformations. Denote $m_R(x,y)=x y$ and $\operatorname{inv}_R(x)=x^{-1}$. Then Diagram \eqref{m^B diagram} and Diagram \ref{inv^B diagram} impliy that for any $R$ in $\mathsf{Alg}_k$ and $x,y \in G(R)$, 
\[
m_R^B(x,y)= x  B_R(x) y B_R(x)^{-1}
\]
and
\[
\mathrm{inv}^B_R(x)= B_R(x)^{-1} x^{-1} B_R(x).
\]
Then \(\big(G(R), m^B_R\big)\) is a group by (1) in Proposition \ref{G oB group}. The inverse of \(x\) is \(\operatorname{inv}^B_R(x)\), and the identity element of \(\big(G(R), m^B_R\big)\) coincides with that of \(\big(G(R), m_R\big)\) by Remark \ref{o^B remark}. Therefore $(G, m^B,e,\mathrm{inv}^B)$ is an affine group.

${\rm (ii)}$ By ${\rm (2)}$ in Proposition \ref{G oB group}, $B_R$ is a RB operator on the group $\big(G(R),m_R^B\big)$. Thus, $(G,m^B,B)$ is  an affine RB group. 

${\rm (iii)}$ By ${\rm (3)}$ in Proposition \ref{G oB group}, $B_R$ is a homomorphism of RB groups from $\big(G(R),m_B^R, B_R \big)$ to $\big(G(R),m_R, B_R \big)$, which yields that $B$ is a homomorphism of affine RB groups from $(G,m^B,B)$ to $(G,m,B)$.
\end{proof}

We now introduce a notion of decomposition for affine groups, which naturally gives rise to an RB operator on the affine group. Recall that an affine subgroup is a subfunctor $H$ of $G$ such that $H(R)$ is a subgroup of $G(R)$ for all $R$ in $\mathsf{Alg}_k$.

\begin{definition}
Let \((G, m)\) be an affine group. An \textbf{exact decomposition} of \(G\) is a pair $(H, L)$, where $H$ and $L$ are affine subgroups of $G$ such that \(G(R) =H(R) L(R)\) and \(H(R) \cap L(R) = \{1_R\}\) for every \(R\).
\end{definition}

\begin{proposition}
Let \((G, m)\) be an affine group equipped with an exact decomposition \((H, L)\). For any \(g \in G(R)\), write uniquely \(g = h \ell\) with \(h \in H(R)\) and \(\ell \in L(R)\). Define
\[
B_R(g) = \ell^{-1}.
\]
Then \((G, m, B)\) is an affine Rota--Baxter group.
\end{proposition}
\begin{proof}
For every \(k\)-algebra \(R\), it is clear that \((G(R), m_R, B_R)\) is an RB group by (3) in Example \ref{prelim examples}. We now prove the naturality of \(B\). Let \(\varphi \colon R\to S\) be any morphism in \(\mathsf{Alg}_k\). For any \(g = h \ell\in G(R)\) with \(h\in H(R), \ell\in L(R)\), the functoriality of \(G\) gives
\[
G(\varphi)(g) = G(\varphi)(h)G(\varphi)(\ell)=H(\varphi)(h) L(\varphi)(\ell).
\]
Moreover, since $H$ and $L$ are subfunctors of $G$, we have $G(\varphi)(h)=H(\varphi)(h) \in H(S)$ and $G(\varphi)(\ell)=L(\varphi)(\ell) \in L(S)$. By assumption, this is the unique factorization of \(G(\varphi)(g)\) in \(G(S)\). Then
\[
B_S\big(G(\varphi)(g)\big) = G(\varphi)(\ell)^{-1},
\]
while
\[
G(\varphi)\big(B_R(g)\big) = G(\varphi)(\ell^{-1}) = G(\varphi)(\ell)^{-1}.
\]
Hence \(B_S\circ G(\varphi) = G(\varphi)\circ B_R\), meaning \(B\) is a natural transformation.
\end{proof}

\begin{example}[Triangular decomposition]
Define the affine group \(G\) as follows: for each \(k\)-algebra \(R\), \(G(R)\) consists of all invertible upper triangular \(n\times n\) matrices with entries in \(R\). This is an affine algebraic group, represented by \(k[X_{ij}, \det^{-1}]\) with relations \(X_{ij}=0\) for all \(i>j\).

For each \(k\)-algebra \(R\), let \(H(R)\) be the subgroup of \(G(R)\) consisting of all invertible diagonal matrices. That is, a matrix in \(H(R)\) has the form \(\text{diag}(a_1,\dots,a_n)\) where each \(a_i\) is a unit in \(R\). Let \(L(R)\) be the subgroup of \(G(R)\) consisting of all upper triangular matrices with $1_R$ on the diagonal.

Then we have an exact factorization \(G(R) = H(R) L(R)\) and the intersection \(H(R) \cap L(R)\) contains only the identity matrix \(I\). Define \(B_R(g) = \ell^{-1}\) for any \(g \in G(R)\), where \(g = h \ell\) is the unique factorization of \(g\) with \(h \in H(R)\) and \(\ell \in L(R)\). One can chect that $B_R$ is natural in $R$. Then \((G, B)\) is an affine algebraic Rota--Baxter group. \qedhere
\end{example}

\subsection{Affine Rota--Baxter groups and Co-Rota--Baxter Hopf algebras}
First we recall the notion of Rota--Baxter co-operators on commutative Hopf algebras.

\begin{definition}{\rm (\cite{ZZL24})}
Let \((H, \mu, \eta, \Delta, \varepsilon, S)\) be a commutative Hopf algebra. An algebra map \(\mathcal{B} : H \to H\) is called a \textbf{Rota--Baxter co-operator} on \(H\) if for all \(x \in H\),
\begin{equation}
\label{RBcooperator}
    \mathcal{B}(x_1) \otimes \mathcal{B}(x_2) = \mathcal{B}(x)_1\mathcal{B}(\mathcal{B}(x)_2S(\mathcal{B}(x)_4)) \otimes \mathcal{B}(x)_3.
\end{equation}
Here we use Sweedler's notation: $\Delta(x)=x_1 \otimes x_2$ and $\Delta^{(3)}(x)=x_1 \otimes x_2 \otimes x_3 \otimes x_4$. By a \textbf{co-Rota--Baxter Hopf algebra} (or simply \textbf{co-RB Hopf algebra}) we mean a pair \((H, \mathcal{B})\) consisting of a commutative Hopf algebra \(H\) and a Rota--Baxter co-operator \(\mathcal{B}\) on \(H\).
\end{definition}

\begin{definition}
    A $\mathbf{homomorphism}$ \(\varphi: (H, \mathcal{B}_H) \to (K, \mathcal{B}_K)\) of co-Rota--Baxter Hopf algebras is a Hopf algebra homomorphism satisfying \(\varphi \circ \mathcal{B}_H =\mathcal{B}_K \circ \varphi\).
\end{definition}
Thus, co-RB Hopf algebras over \(k\) together with their homomorphisms form a category, denoted by \(\mathsf{CoRBHopf}_k\).

\begin{theorem}
\label{affineRBG RBHopf THM}
    Let $(H,\mu,\eta ,\Delta, \varepsilon, S)$ be a commutative Hopf algebra. If $(H,\mathcal{B})$ is co-Rota--Baxter Hopf algebra, then $(h^H,h^\Delta,h^{\mathcal{B}})$ is an affine Rota--Baxter group. Conversely, if $(G,m,B)$ is an affine Rota--Baxter group, $\mathcal{O}(G,m,B)$ is a co-Rota--Baxter Hopf algebra.
\end{theorem}
\begin{proof}
    Let $(G,m,B)$ be an affine RB group with natural transformations $e$ and $\operatorname{inv}$. Then
    \[
    G\times G \simeq h^{\mathcal{O}(G)} \times h^{\mathcal{O}(G)} \simeq h^{\mathcal{O}(G) \otimes \mathcal{O}(G)}.
    \]
Moreover $m$, $e$, $\operatorname{inv}$ and $B$ correspond to comultiplication $\Delta$, counit $\varepsilon$, antipode $S$ and an algebra morphism $\mathcal{B}$, respectively, such that $(\mathcal{O}(G), \Delta, \varepsilon, S)$ is a commutative Hopf algebra. All that remains is to check that \((\mathcal{O}(G), \mathcal{B})\) satisfies the equation \eqref{RBcooperator}. According to Theorem \ref{affineRBgroup THM}, $(G,m,B)$ is an RB group object in $\mathsf{RepFunc}[\mathsf{Alg}_k, \mathsf{Set}]$, which implies the following commutative diagram in $\mathsf{RepFunc}[\mathsf{Alg}_k, \mathsf{Set}]$:
   \begin{equation}
   \label{G affineRBGroup}
    \begin{tikzcd}[column sep = 2.5cm, row sep = huge]
     G^{\times 2} \ar[r,"B\times B"] \ar[d,"{\langle \pi_1, \pi_1, \pi_2, \pi_1 \rangle}"]  & G^{\times 2} \ar[r,"m"]  & G \\G^{\times 4} \ar[r,"\mathrm{id} \times B \times \mathrm{id} \times (\mathrm{inv} \circ B)"]  & G^{\times 4} \ar[r,"m^{(3)}"] & G \ar[u,"B"],
    \end{tikzcd}
   \end{equation}
where $m^{(3)}$ is given by \eqref{phi(n):A(n+1) to A}.
As the contravariant equivalence between $\mathsf{RepFunc}[\mathsf{Alg}_k, \mathsf{Set}]$ and $\mathsf{Alg}_k$, Diagram \eqref{G affineRBGroup} is equivalent to the following commutative diagram in $\mathsf{Alg}_k$:
   \begin{equation}
   \label{O(G)RBHOPFalg}
    \begin{tikzcd}[column sep = 2cm, row sep = huge]
     \mathcal{O}(G)^{\otimes 2} & \mathcal{O}(G)^{\otimes 2} \ar[l,"\mathcal{B}\otimes \mathcal{B}"]  & \mathcal{O}(G) \ar[l,"\Delta"] \ar[d, "\mathcal{B}"]
     \\\mathcal{O}(G)^{\otimes 4} \ar[u,"\mu_{\mathcal{O}(G) \otimes \mathcal{O}(G)}^{(3)} \circ (\iota_1 \otimes \iota_1 \otimes \iota_2 
     \otimes \iota_1)"]  & \mathcal{O}(G)^{\otimes 4} \ar[l,"\mathrm{id} \otimes \mathcal{B} \otimes \mathrm{id} \otimes (\mathcal{B}\circ S)"]  & \mathcal{O}(G) \ar[l,"\Delta^{(3)}"] ,
    \end{tikzcd}
   \end{equation}
where $\Delta^{(3)}$ is given by \eqref{psi(n):A to A(n+1)} and $\mu_{\mathcal{O}(G) \otimes \mathcal{O}(G)}^{(3)}$ is given by \eqref{psi(n):A(n+1) to A}. Here $\mu_{\mathcal{O}(G) \otimes \mathcal{O}(G)}: \mathcal{O}(G)^{\otimes 4} \to \mathcal{O}(G)^{\otimes 2}$ denotes the multiplication of $\mathcal{O}(G) \otimes \mathcal{O}(G)$, which is induced by the multiplication of $\mathcal{O}(G)$ and
\[
\left\{
\begin{aligned}
    \mu_{\mathcal{O}(G) \otimes \mathcal{O}(G)}^{(3)} \circ (\iota_1 \otimes \iota_1 \otimes \iota_2 
     \otimes \iota_1)(a\otimes b \otimes c \otimes d)&=abd \otimes c,\\
    \mathrm{id} \otimes \mathcal{B} \otimes \mathrm{id} \otimes (\mathcal{B}\circ S)(a\otimes b \otimes c \otimes d)&=a \otimes \mathcal{B}(b) \otimes c \otimes \mathcal{B}S(d).\\
\end{aligned}
\right.
\]
So Diagram \eqref{O(G)RBHOPFalg} is equivalent to the equation \eqref{RBcooperator}. Therefore, $(\mathcal{O}(G), \mathcal{B})$ is a co-RB Hopf algebra.

Conversely, let $(H,\mu,\eta ,\Delta, \varepsilon, S)$ be a commutative Hopf algebra and $(H,\mathcal{B})$ be a co-RB Hopf algebra. Then $\Delta:H \to H \otimes H$ and $\mathcal{B}:H \to H$ correspond to natural transformations $m:=h^\Delta: h^H \times h^H \to h^H$ and $B:=h^{\mathcal{B}}: h^H \to h^H$ such that $(h^H, m)$ is an affine group. We will show that $(h^H,m, B)$ is an affine RB group. Note that the equation \eqref{RBcooperator} is equivalent to the following commutative diagram in $\mathsf{Alg}_k$:
\begin{equation}
\label{H is RBHalg}
    \begin{tikzcd}[column sep = 2.5cm, row sep = huge]
     H^{\otimes 2} & H^{\otimes 2} \ar[l,"\mathcal{B}\otimes \mathcal{B}"]  & H \ar[l,"\Delta"] \ar[d, "\mathcal{B}"]
     \\H^{\otimes 4} \ar[u,"\mu_{H\otimes H}^{(3)} \circ (\iota_1 \otimes \iota_1 \otimes \iota_2 \otimes \iota_1)"]  & H^{\otimes 4} \ar[l,"\mathrm{id} \otimes \mathcal{B} \otimes \mathrm{id} \otimes (\mathcal{B} \circ S)"]  & H \ar[l,"\Delta^{(3)}"] ,
    \end{tikzcd}
   \end{equation}
where $\mu_{H\otimes H}^{(3)}$ is given by \eqref{psi(n):A(n+1) to A} and $\Delta^{(3)}$ is given by \eqref{psi(n):A to A(n+1)}. By Yoneda lemma Diagram \eqref{H is RBHalg} is equivalent to the following commutative diagram in $\mathsf{RepFunc}[\mathsf{Alg}_k, \mathsf{Set}]$:
\begin{equation}
    \begin{tikzcd}[column sep = 2.5cm, row sep = huge]
     (h^H)^{\times 2} \ar[r,"B\times B"] \ar[d,"{\langle \pi_1, \pi_1, \pi_2, \pi_1 \rangle}"]  & (h^H)^{\times 2} \ar[r,"m"]  & h^H \\(h^H)^{\times 4} \ar[r,"{\mathrm{id} \times B \times \mathrm{id} \times (\mathrm{inv} \circ B)}"]  & (h^H)^{\times 4} \ar[r,"m^{(3)}"] & h^H \ar[u,"B"] ,
    \end{tikzcd}
   \end{equation}
where $m^{(3)}$ is given by \eqref{phi(n):A(n+1) to A}.
Thus, $(h^H,m, B)$ is an RB group object in $\mathsf{RepFunc}[\mathsf{Alg}_k, \mathsf{Set}]$, i.e., $(h^H,m,B)$ is an affine RB group.
\end{proof}

\begin{theorem}
\label{coRBHopf = AffRBgroup category equivalence}
The above correspondence is functorial and yields a contravariant equivalence of categories:
\[
\mathsf{CoRBHopf}^{\mathrm{opp}}_k \simeq \mathsf{AffRBGrp}_k.
\]
\end{theorem}
\begin{proof}
Note that the functors
\[
\mathcal{O}: \mathsf{AffRBGrp}_k \longrightarrow \mathsf{CoRBHopf}^{\operatorname{opp}}_k, \qquad (G,m, B) \longmapsto \mathcal{O}(G,m,B)=\big(\mathcal{O}(G),\mathcal{O}(m),\mathcal{O}(B) \big)
\]
and
\[
H \rightsquigarrow h^H: \mathsf{CoRBHopf}_k \longrightarrow \mathsf{AffRBGrp}^{\operatorname{opp}}_k, \qquad (H, \Delta,\mathcal{B}) \longmapsto (h^H,h^\Delta, h^{\mathcal{B}})
\]
form an anti-equivalence of categories. 

We first verify that \(\mathcal{O}\) and \(H \rightsquigarrow h^H\) are well-defined functors. For an affine Rota--Baxter group \((G,m, B)\), the coordinate ring \(\mathcal{O}(G)\) is a co-RB Hopf algebra by Theorem \ref{affineRBG RBHopf THM}. 
For a homomorphism \(\Psi: (G,m_G, B_G) \to (H,m_H, B_H)\) of affine RB groups, the Yoneda lemma together with Proposition \ref{Prop affinegroup equivalent  to Hopf} imply the existence of a unique Hopf algebra homomorphism
\[
\psi: \mathcal{O}(H) \longrightarrow \mathcal{O}(G)
\]
corresponding to \(\Psi\). Moreover, the condition \(B_H \circ \Psi = \Psi \circ B_G\) implies that
\[
\psi \circ \mathcal{O}(B_H) = \mathcal{O}(B_G) \circ \psi.
\]
Thus \(\psi\) is a homomorphism of co-RB Hopf algebras. The functoriality conditions
\[
\mathcal{O}(\operatorname{id}_G) = \operatorname{id}_{\mathcal{O}(G)}, \qquad \mathcal{O}(g \circ f) = \mathcal{O}(f) \circ \mathcal{O}(g) 
\]
are immediate consequences of the Proposition \ref{Prop affinegroup equivalent  to Hopf}.
Thus the correspondence defines a functor \(\mathcal{O}: \mathsf{AffRBGrp}_k \to \mathsf{CoRBHopf}^{\operatorname{opp}}_k\). 

Conversely, for a co-RB Hopf algebra \((H,\Delta, \mathcal{B})\), we have an affine Rota--Baxter group \((h^H,h^\Delta, h^{\mathcal{B}})\) by Theorem \ref{affineRBG RBHopf THM}. 
A co-RB Hopf algebra homomorphism
\(\theta: (H, \Delta_H, \mathcal{B}_H) \to (K, \Delta_K, \mathcal{B}_K)\)
is a Hopf algebra homomorphism \(\theta: H \to K\) such that
\[
\mathcal{B}_K \circ \theta = \theta \circ \mathcal{B}_H.
\]
Thus $\theta$ induces a natural transformation \(\Theta=h^{\theta}: h^K \to h^H\) where 
\[
\Theta_R: h^K(R) \to h^H(R), \quad f \mapsto f \circ \theta.
\]
Then we have $\Theta_R \circ (h^{\mathcal{B}_K})_R (f)=f\circ \mathcal{B}_K \circ \theta = f \circ \theta \circ \mathcal{B}_H = (h^{\mathcal{B}_H})_R \circ \Theta_R (f)$, namely, 
\[
\Theta \circ h^{\mathcal{B}_K} = h^{\mathcal{B}_H} \circ \Theta.
\]
So $H \rightsquigarrow h^H$ is also well-defined.

Moreover, we have natural isomorphisms \(H\rightsquigarrow h^H\circ\mathcal{O} \cong \operatorname{id}\) and \(\mathcal{O} \circ H\rightsquigarrow h^H \cong \operatorname{id}\) by Proposition \ref{Prop affinegroup equivalent  to Hopf} and Yoneda lemma. This gives a contravariant equivalence of categories.
\end{proof}

\subsection{Descendant affine groups and descendant Hopf algebras}
Given an affine RB group \((G, m, e, \operatorname{inv}, B)\), we have constructed its descendant affine group \((G, m^B, e, \operatorname{inv}^B)\) in Proposition~\ref{descend affine RBG}. This yields a new affine group structure on the same underlying functor \(G\), thereby inducing a new commutative Hopf algebra associated with the original co-RB Hopf algebra \((H, \mathcal{B})\). The corresponding comultiplication and antipode, denoted by \(\Delta_{\mathcal{B}}\) and \(S_{\mathcal{B}}\), are corresponding to the natural transformations \(m^B\) and \(\operatorname{inv}^B\), respectively.

We now obtain explicit expressions for \(\Delta_{\mathcal{B}}\) and \(S_{\mathcal{B}}\) by dualizing the corresponding diagrams \eqref{m^B diagram} and \eqref{inv^B diagram} via the Yoneda lemma:
    \begin{equation}
    \label{Delta_B def}
     \begin{tikzcd}[column sep = 4cm, row sep = huge]
     H^{\otimes 2}  
     & H \ar[l, "\Delta_{\mathcal{B}}"] \ar[d, "\Delta^{(3)}"]
     \\
     H^{\otimes 4} \ar[u, "\mu_{H\otimes H}^{(3)} \circ (\iota_1 \otimes \iota_1 \otimes \iota_2 \otimes \iota_1)"]
     & H^{\otimes 4} \ar[l, "\operatorname{id} \otimes \mathcal{B} \otimes \operatorname{id} \otimes (\mathcal{B} \circ S)"] 
     \end{tikzcd}
    \end{equation}
and
\begin{equation}
\label{S_B def}
     \begin{tikzcd}[column sep = 4cm, row sep = huge]
     H 
     & H \ar[l, "S_{\mathcal{B}}"] \ar[d, "\Delta^{(2)}"]
     \\
     H^{\otimes 3} \ar[u, "\mu_H^{(2)} \circ (\operatorname{id}_H \otimes \operatorname{id}_H \otimes \operatorname{id}_H)"]
     & H^{\otimes 3} \ar[l, "(\mathcal{B} \circ S) \otimes S \otimes \mathcal{B}"] .
     \end{tikzcd}
    \end{equation}
Here $\mu_{H\otimes H}^{(3)}$ and $\mu^{(2)}_H$ are given by \eqref{psi(n):A(n+1) to A}, and $\Delta^{(2)}$ and $\Delta^{(3)}$ are given by \eqref{psi(n):A to A(n+1)}.

The commutativity of the above diagrams implies that
\begin{equation}
\label{descendant comul}
    \Delta_{\mathcal{B}}(x) = x_1 \mathcal{B}(x_2) \mathcal{B}(S(x_4)) \otimes x_3
\end{equation}
and
\begin{equation}
\label{descedant antipo}
    S_{\mathcal{B}}(x)=\mathcal{B}(S(x_1)) S(x_2) \mathcal{B}(x_3).
\end{equation}

\begin{definition}
\label{def descendant hopf alg}
    Let $(H,\mu,\eta,\Delta,\varepsilon,S,\mathcal{B})$ be a co-RB Hopf algebra. Then $H_\mathcal{B}:=(H,\mu,\eta,\Delta_\mathcal{B},\varepsilon, S_\mathcal{B})$, where $\Delta_\mathcal{B}$ and $S_\mathcal{B}$ are given by \eqref{descendant comul} and \eqref{descedant antipo}, is called the \textbf{descendant Hopf algebra}.
\end{definition}

\begin{remark}
    Proposition \ref{descend affine RBG} shows that $B$ is a Rota--Baxter operator on the descendant affine group $(G,m^B,e,\operatorname{inv}^B)$, which implies that $\mathcal{B}$ is again a Rota--Baxter co-operator on $H_\mathcal{B}$.
\end{remark}

In summary, let $(G,m^B,e,\operatorname{inv}^B)$ be the descendant affine group of the affine RB group $(G,m,e,\operatorname{inv},B)$. Then $\mathcal{O}(G,m^B)$ is called the descendant Hopf algebra of the co-RB Hopf algebra $\mathcal{O}(G,m)$. Conversely, let $H_{\mathcal{B}}$ be the descendant Hopf algebra of the co-RB Hopf algebra $(H, \mu, \eta, \Delta, \varepsilon, S,\mathcal{B})$. Then $(h^H, h^{\Delta_{\mathcal{B}}})$ is the descendant affine group of the affine RB group $(h^H,h^{\Delta}, h^{\mathcal{B}})$.

\begin{remark}
There is another way to construct a new Hopf algebra from a co-RB Hopf algebra in \cite{ZZL24}. By adopting a categorical viewpoint, the method presented in our paper avoids lengthy computations.
\end{remark}

\section{Affine (skew) braces and Hopf co-braces}\label{sec:brace}
The notion of an affine brace was introduced by Angiono, Galindo, and Vendramin \cite{AGV17}. In their definition, the first group operation is not required to be abelian. In this paper, we adopt a more refined terminology: we call such a structure an \textbf{affine skew brace}, and reserve the term \textbf{affine brace} for the special case where the first group is abelian. Accordingly, the correspondence of \cite{AGV17} can be rephrased as: affine skew braces correspond to commutative Hopf co-braces. As a special case, we then obtain that affine braces correspond to commutative Hopf co-braces whose first comultiplication is cocommutative.

\subsection{(Skew) brace objects}
\begin{definition}
    A \textbf{skew brace object} in \(\mathcal{C}\) is a triple \((A, m_\bullet, m_\circ)\) consisting of 
    an object \(A \in \mathcal{C}\) and two morphisms
    \[
    m_\bullet : A \times A \to A \quad\text{and}\quad m_\circ : A \times A \to A
    \]
    such that the following conditions hold:
    \begin{enumerate}
        \item[\rm(1)] \((A, m_\bullet)\) is a group object in \(\mathcal{C}\). 
              We denote its unit by \(e_\bullet : * \to A\) and its inverse by 
              \(\mathrm{inv}_\bullet : A \to A\).
        \item[\rm(2)] \((A, m_\circ)\) is a group obejct in \(\mathcal{C}\). 
              We denote its unit by \(e_\circ : * \to A\) and its inverse by 
              \(\mathrm{inv}_\circ : A \to A\).
        \item[\rm(3)] \((\textrm{Compatibility condition}) \) The following diagram commutes:
       \begin{equation}
       \label{skewbracediagram}
           \begin{tikzcd}[column sep=5cm, row sep=huge]
            A^{\times 3} 
            \ar[r,"{\big\langle m_\circ \circ \langle \pi_1,\pi_2\rangle,\;\mathrm{inv}_\bullet \circ \pi_1,\; m_\circ \circ \langle \pi_1,\pi_3\rangle \big\rangle}"]
            \ar[d,"\operatorname{id}_A \times m_\bullet"]
            & A^{\times 3} \ar[d,"m_\bullet^{(2)}"] \\
            A^{\times 2} \ar[r,"m_\circ"] & A ,
           \end{tikzcd}
       \end{equation}
    \end{enumerate}
    where $m_\bullet^{(2)}$ is given by \eqref{phi(n):A(n+1) to A}.
    Here \(\pi_i : A^{\times 3}\) (\(i = 1,2,3\)) denote the canonical projections.
    When all structure morphisms are needed, we write the septuple $(A, m_\bullet, e_\bullet, \operatorname{inv}_\bullet, m_\circ, e_\circ, \operatorname{inv}_\circ)$.
\end{definition}

\begin{remark}
    The compatibility condition expressed by the diagram above is equivalent to the equality of 
    the following two composite morphisms from \(A ^{\times 3}\) to \(A\):
    \[
    m_\circ \circ (\operatorname{id}_A \times m_\bullet) 
    = 
    m_\bullet \circ (m_\bullet \times \mathrm{id}_A) \circ \big\langle m_\circ \circ \langle \pi_1, \pi_2\rangle,\; \mathrm{inv}_\bullet \circ \pi_1,\; m_\circ \circ \langle \pi_1, \pi_3\rangle \big\rangle .
    \]
    This is precisely the categorical translation of the defining equation of a skew brace
    in the category of sets:
    \[
    a \circ (b \cdot c) = (a \circ b) \cdot a^{-1} \cdot (a \circ c),
    \]
    where \(\cdot\) corresponds to \(m_\bullet\), \(\circ\) corresponds to \(m_\circ\), and 
    \(a^{-1} = \mathrm{inv}_\bullet(a)\).
\end{remark}

If, in addition, the first group object \((A, m_\bullet)\) is required to be abelian, we obtain the definition of a brace object.

\begin{definition}
A \textbf{brace object} in \(\mathcal{C}\) is a skew brace object in $\mathcal{C}$ denoted by \((A, m_\bullet, m_\circ)\) and additionally \((A, m_\bullet)\) is an \textbf{abelian group object}, i.e. the following diagram commutes:
\begin{equation}
\begin{tikzcd}
A \times A \arrow[r, "\tau"] \arrow[d, "m_\bullet"'] & A \times A \arrow[d, "m_\bullet"] \\
A \arrow[r, "\mathrm{id}_A"] & A ,
\end{tikzcd}
\end{equation}
where \(\tau : A \times A \to A \times A\) is the swap, i.e., \(\tau = \langle\pi_2, \pi_1\rangle \).
\end{definition}

Our framework provides a uniform categorical treatment of various existing structures.

\begin{example}
    A (skew) brace object in $\mathsf{Set}$ is exactly a (skew) brace.
\end{example}

\begin{example}
A skew brace object in the category of smooth manifolds consists of \(G\) together with two smooth group structures \((G, \cdot)\) and \((G, \circ)\) such that the compatibility condition holds. This is precisely the definition of a \textbf{Lie skew brace} as introduced in \cite{DL26}.
\end{example}

\subsection{Affine (skew) braces}
\begin{definition}{\rm (\cite{AGV17})}
    An \textbf{affine (skew) brace scheme over $k$} (or simply an \textbf{affine skew brace}) is a representable functor $A \colon \mathsf{Alg}_k \to \mathsf{Set}$ together with two natural transformations \(m_\bullet, m_\circ : \mathcal{A} \times \mathcal{A} \to \mathcal{A}\) such that, for all $R$ in $\mathsf{Alg}_k$, 
\[
{m_\bullet}_R : \mathcal{A}(R) \times \mathcal{A}(R) \to \mathcal{A}(R),\quad (a, b) \mapsto a \bullet_R b,
\]
and
\[
{m_\circ}_R : \mathcal{A}(R) \times \mathcal{A}(R) \to \mathcal{A}(R),\quad (a, b) \mapsto a \circ_R b,
\]
   make \(\big(\mathcal{A}(R), \bullet_R, \circ_R \big)\) into a (skew) brace. If $A$ is represented by a finitely generated $k$-algebra, then it is called an \textbf{affine algebraic (skew) brace}.
\end{definition}

\begin{remark}
The affine skew brace defined above corresponds to the affine brace \cite[Definition 5.1]{AGV17}. We therefore reserve the term "affine brace" specifically for the abelian setting.
\end{remark}

\begin{definition}
Let \((\mathcal{A}, m_\bullet, m_\circ)\) and \((\mathcal{A}', m_\bullet', m_\circ')\) be affine (skew) braces over \(k\). A \textbf{homomorphism} \(\Phi: \mathcal{A} \to \mathcal{A}'\) is a natural transformation such that for every \(k\)-algebra \(R\), the map
\[
\Phi_R: \mathcal{A}(R) \longrightarrow \mathcal{A}'(R)
\]
is a homomorphism of (skew) braces. Equivalently, the following diagrams commute:
\begin{equation}
\begin{tikzcd}
\mathcal{A} \times \mathcal{A} \arrow[r, "m_\bullet"] \arrow[d, "\Phi \times \Phi"'] & \mathcal{A} \arrow[d, "\Phi"] \\
\mathcal{A}' \times \mathcal{A}' \arrow[r, "m_\bullet'"] & \mathcal{A}'
\end{tikzcd}
\qquad
\begin{tikzcd}
\mathcal{A} \times \mathcal{A} \arrow[r, "m_\circ"] \arrow[d, "\Phi \times \Phi"'] & \mathcal{A} \arrow[d, "\Phi"] \\
\mathcal{A}' \times \mathcal{A}' \arrow[r, "m_\circ'"] & \mathcal{A}' .
\end{tikzcd}
\end{equation}
We say $\Phi$ is {\bf injective} if $\Phi_R$ is injective for all $k$-algebras $R$. An {\bf embedding} is an injective homomorphism.
\end{definition}

Thus, affine (skew) braces over \(k\) together with their homomorphisms form a category, denoted by \(\mathsf{AffSBr}_k\) (resp. \(\mathsf{AffBr}_k\)).  

\begin{definition}
    Let \((\mathcal{A}, m_\bullet, m_\circ)\) be an affine (skew) brace over \(k\). If $R\in \mathsf{Alg}_k$ represents \(\mathcal{A}\), then \(R\) is called the \textbf{coordinate ring} of \(\mathcal{A}\) and is denoted by \(\mathcal{O}(\mathcal{A})\) or \(\mathcal{O}(\mathcal{A}, m_\bullet, m_\circ)\).
\end{definition}

The following theorem gives equivalent characterizations of affine (skew) braces.

\begin{theorem}
\label{affine (skew) brace equivalent thm}
    The following are equivalent:
\begin{enumerate}
    \item[\rm(1)] \((\mathcal{A}, m_\bullet, m_\circ)\) is an affine (skew) brace over $k$.
    \item[\rm(2)] \((\mathcal{A}, m_\bullet, m_\circ)\) is a (skew) brace object in $\mathsf{RepFunc}[\mathsf{Alg}_k,\mathsf{Set}]$.
    \item[\rm(3)] A functor $\mathcal{A}$ from $\mathsf{Alg}_k$ to the category of (skew) braces whose underlying set-valued functor is representable.
\end{enumerate}
\end{theorem}
\begin{proof}
In what follows, we prove the statement for skew braces; the brace version follows by a similar argument. 

First we show the equivalence between (1) and (2). By Proposition~\ref{pro:equivalent affine groups}, the affine group $(\mathcal{A},m_\bullet)$ and  $(\mathcal{A},m_\circ)$ are equivalent to group objects $(\mathcal{A},m_\bullet,e_\bullet,\operatorname{inv}_\bullet)$ and $(\mathcal{A},m_\circ,e_\circ,\operatorname{inv}_\circ)$ in $\mathsf{RepFunc}[\mathsf{Alg}_k, \mathsf{Set}]$. For $(\mathcal{A},m_\bullet,m_\circ)$ to be an affine skew brace, the following diagram must commute for every $R\in\mathsf{Alg}_k$:
     \[
           \begin{tikzcd}[column sep=5cm, row sep=huge]
            \mathcal{A}(R)^{\times 3} 
            \ar[r,"{\big\langle {m_\circ}_R \circ \langle \pi_1,\pi_2 \rangle,\;(\mathrm{inv}_\bullet)_R \circ \pi_1,\; {m_\circ}_R \circ \langle \pi_1,\pi_3 \rangle \big \rangle}"]
            \ar[d,"\mathrm{id} \times {m_\bullet}_R"]
            & \mathcal{A}(R)^{\times 3} \ar[d,"{m_\bullet}_R^{(2)}"] \\
            \mathcal{A}(R)^{\times 2} \ar[r,"{m_\circ}_R"] & \mathcal{A}(R) ,
           \end{tikzcd}
        \]
    where ${m_\bullet}_R^{(2)}$ is given by \eqref{phi(n):A(n+1) to A}.
    And this is equivalent to $(\mathcal{A},m_\bullet,e_\bullet,\operatorname{inv}_\bullet)$ and $(\mathcal{A},m_\circ,e_\circ,\operatorname{inv}_\circ)$ are group objects in $\mathsf{RepFunc}[\mathsf{Alg}_k, \mathsf{Set}]$ and satisfy Diagram \eqref{skewbracediagram}, i.e., $(\mathcal{A},m_\bullet,m_\circ)$ is a skew brace object in $\mathsf{RepFunc}[\mathsf{Alg}_k, \mathsf{Set}]$. 

Next we show (1) is equivalent to (3). Let $(\mathcal{A},m_\bullet,m_\circ)$ be an affine skew brace. The naturality of $m_\bullet$ and $m_\circ$ ensures that for any algebra homomorphism $\varphi: R \to R'$, the induced map $\mathcal{A}(\varphi): \mathcal{A}(R) \to \mathcal{A}(R')$ preserves the multiplications $m_{\bullet R}$ and $m_{\circ R}$; consequently $\mathcal{A}(\varphi)$ is a homomorphism of skew braces. This defines a functor $\mathsf{Alg}_k \to \mathsf{SBr}$. 

Conversely, note that $\mathcal{A}$ is representable by assumption. And for any $R\in \mathsf{Alg}_k$, there is a skew brace structure over $\mathcal{A}(R)$, denoted by $\big( \mathcal{A}(R), \cdot, \circ \big)$. Let ${m_\bullet}_R(f,g):=f \cdot g$ and ${m_\circ}_R:=f \circ g$. It remains to show that ${m_\bullet}_R$ and ${m_\circ}_R$ are natural in $R$. For any algebra homomorphism $\varphi: R \to S$, $\mathcal{A}(\varphi)$ is a skew brace homomorphism from $\mathcal{A}(R)$ to $\mathcal{A}(S)$. Equivalently, the following diagrams commute:
\[
\begin{tikzcd}
\mathcal{A}(R) \times \mathcal{A}(R) \arrow[r, "{m_\bullet}_R"] \arrow[d, "\mathcal{A}(\varphi) \times \mathcal{A}(\varphi)"'] & \mathcal{A}(R) \arrow[d, "\mathcal{A}(\varphi)"] \\
\mathcal{A}(S) \times \mathcal{A}(S) \arrow[r, "{m_\bullet}_S"'] & \mathcal{A}(S)
\end{tikzcd}
\qquad
\begin{tikzcd}
\mathcal{A}(R) \times \mathcal{A}(R) \arrow[r, "{m_\circ}_R"] \arrow[d, "\mathcal{A}(\varphi) \times \mathcal{A}(\varphi)"'] & \mathcal{A}(R) \arrow[d, "\mathcal{A}(\varphi)"] \\
\mathcal{A}(S) \times \mathcal{A}(S) \arrow[r, "{m_\circ}_S"'] & \mathcal{A}(S) .
\end{tikzcd}
\]
This implies that $m_\bullet$ and $m_\circ$ are natural transformations from $\mathcal{A} \times \mathcal{A}$ to $\mathcal{A}$. Thus, the proof is complete.
\end{proof}

In the language of skew brace objects, to give a complete description of the structure of an affine skew brace \((\mathcal{A}, m_\bullet, m_\circ)\), we may use the septuple
\[
(\mathcal{A}, m_\bullet, e_\bullet, \operatorname{inv}_\bullet, m_\circ, e_\circ, \operatorname{inv}_\circ),
\]
where \((\mathcal{A}, m_\bullet, e_\bullet, \operatorname{inv}_\bullet)\) and \((\mathcal{A}, m_\circ, e_\circ, \operatorname{inv}_\circ)\) are affine groups. However we will show that \(e_\bullet\) and \(e_\circ\) coincide.

\begin{proposition}
\label{affine skewbrace common unit}
    For any affine skew brace $(\mathcal{A}, m_\bullet, e_\bullet, \operatorname{inv}_\bullet, m_\circ, e_\circ, \operatorname{inv}_\circ)$, the unit morphisms coincide.
\end{proposition}
\begin{proof}
    By Lemma \ref{skew brace lemma}, for every $k$-algebra $R$ groups $\big(\mathcal{A}(R), \cdot \big)$ and $\big(\mathcal{A}(R), \circ \big)$ have the same identity element. So the image of ${e_\bullet}_R: h^k(R) \to \mathcal{A}(R)$ is equal to the image of ${e_\circ}_R: h^k(R) \to \mathcal{A}(R)$, i.e., ${e_\bullet}_R(\eta_R)={e_\circ}_R(\eta_R)$. Thus \(e_\bullet = e_\circ\).
\end{proof}

\subsection{Affine (skew) braces and commutative Hopf co-braces}

\begin{definition}{\rm (\cite{AGV17})}
\label{def Hopf cobrace}
Let \((A, \mu, \eta)\) be a $k$-algebra, which is not necessary commutative. A \textbf{Hopf co-brace structure} over \(A\) consists of the following data:
\begin{enumerate}
    \item[\rm(1)] a Hopf algebra structure \((A, \mu, \eta, \Delta_\bullet, \epsilon_\bullet, S_\bullet)\) and
    \item[\rm(2)] a Hopf algebra structure \((A, \mu, \eta, \Delta_\circ, \epsilon_\circ, S_\circ)\)
\end{enumerate}
satisfying the following compatibility:
\begin{equation}
\label{hopfcobrace}
a_{1_\circ} \otimes (a_{2_\circ})_{1_\bullet} \otimes (a_{2_\circ})_{2_\bullet} = (a_{1_\bullet})_{1_\circ} S(a_{2_\bullet})(a_{3_\bullet})_{1_\circ} \otimes (a_{1_\bullet})_{2_\circ} \otimes (a_{3_\bullet})_{2_\circ}
\end{equation}
for all \(a \in A\). Here we use Sweedler's notation: $\Delta_\circ(a) = a_{1_\circ} \otimes a_{2_\circ}$, 
$\Delta_\bullet(a) = a_{1_\bullet} \otimes a_{2_\bullet}$, and $\Delta_\bullet^{(2)}(a) = a_{1_\bullet} \otimes a_{2_\bullet} \otimes a_{3_\bullet}$.
And by a triple \((A, \Delta_\bullet, \Delta_\circ)\) we mean a Hopf co-brace.
\end{definition}

A Hopf co-brace \((A, \Delta_\bullet, \Delta_\circ)\) is said to be \textbf{commutative} if \(A\) is a commutative algebra. 

Recall that a \textbf{homomorphism of Hopf co-braces} $\psi: (H,\Delta_\bullet, \Delta_\circ) \to (G,\Delta_\bullet, \Delta_\circ)$ is a linear map such that $\psi: (H,\Delta_\bullet) \to (G,\Delta_\bullet)$ and $\psi: (H,\Delta_\circ) \to (G,\Delta_\circ)$ are Hopf algebra homomorphisms.

We denote by \(\mathsf{HopfCoBr}_k\) the category of commutative Hopf co-braces.

Proposition 5.6 of \cite{AGV17} can be reformulated in our terminology as follows: a commutative Hopf co-brace corresponds to an affine skew brace. Moreover, building on their result, we further obtain that an affine brace corresponds to a commutative Hopf co-brace whose first comultiplication is cocommutative.

\begin{theorem}{\rm(\cite{AGV17})}
\label{affine skewbrace Hopf co-brace}
If \(A\) is a commutative Hopf co-brace then \(h^A\) is an affine skew brace. Conversely, if \(\mathcal{A}\) is an affine skew brace, then \(\mathcal{O}(\mathcal{A})\) is a commutative Hopf co-brace. 

Moreover,  this correspondence is functorial and yields a contravariant equivalence of categories:
\[
\mathsf{AffSBr}_k \simeq \mathsf{HopfcoBr}_k^{\mathrm{opp}}.
\]
\end{theorem}

\begin{corollary}
    Affine braces \((\mathcal{A}, m_\bullet, m_\circ)\) correspond to commutative Hopf co-braces $(A, \Delta_\bullet, \Delta_\circ)$ whose first comultiplication $\Delta_\bullet$ is cocommutative. 
\end{corollary}
\begin{proof}
    It follows from the duality between the commutativity of the multiplication $m_\bullet$ of an affine brace and the cocommutativity of the comultiplication $\Delta_\bullet$ of a Hopf co-brace. 
\end{proof}

\section{Affine Rota--Baxter groups, affine skew braces and their coordinate rings}\label{sec:correspond}
In this section, we investigate the connections between affine RB groups and affine skew braces. We show that every affine RB group induces an affine skew brace, and conversely, every affine skew brace admits an embedding into an affine RB group. Finally, we establish an equivalence between co-RB Hopf algebras and Hopf co-braces
\subsection{Affine Rota--Baxter groups and affine skew braces}
Inspired by the classical correspondence between RB groups and skew braces, we establish an analogous correspondence between affine RB groups and affine skew braces.

\begin{proposition}
\label{affine RB to affine skew brace prop}
    Let \((G, m, e, \operatorname{inv}, B)\) be an affine RB group and \((G, m^B, e, \operatorname{inv}^B)\) its descendant affine group. Then \((G, m, m^B)\) is an affine skew brace.
\end{proposition}
\begin{proof}
For every \(k\)-algebra \(R\), the pairs \(\big(G(R), m_R \big)\) and \( \big(G(R), m_R^B \big)\) are groups, where \(m_R(x, y) = x y\) and $m_R^B(x, y)= x B_R(x) y B(x)^{-1}.$ Then \( \big(G(R), m_R, m^B_R \big)\) is a skew brace by Proposition \ref{RBgroup to skew brace  proposition}. Hence \(\big(G, m, m^B\big)\) is an affine skew brace. \qedhere
\end{proof}

Let $(\mathcal{A}, m_\bullet, \operatorname{inv}_\bullet, e, m_\circ, \operatorname{inv}_\circ)$ be an affine skew brace.  For each $R=(R,\mu_R,\eta_R)$, we denote the multiplications by
\[
x \cdot_R y=m_{\bullet R}(x,y)  ,\quad x \circ_R y =m_{\circ R}(x,y) .
\]
Let $\operatorname{inv}_{\bullet R}(x)=x^{-1}$ denote the inverse of $x$ in the group $\bigl(\mathcal{A}(R), \cdot_R\bigr)$, and let $\operatorname{inv}_{\circ R}(x)=x^{\circ-1}$ denote the inverse in $\bigl(\mathcal{A}(R), \circ_R\bigr)$.
The two group structures share a common identity element, which we denote by $1_R = e_R(\eta_R)$.

\begin{lemma}
\label{Lem affine skeW brace give affine RB group}
Let \((\mathcal{A}, m_\bullet, m_\circ)\) be an affine skew brace over \(k\). 
Define a functor \(\tilde{\mathcal{A}}: \mathsf{Alg}_k \to \mathsf{Set}\) as follows:
\begin{itemize}
    \item For each \(R \in \mathsf{Alg}_k\), $\tilde{\mathcal{A}}(R) = \big(\mathcal{A}(R), {m_\circ}_R \big) \ltimes \big(\mathcal{A}(R), {m_\bullet}_R \big)$.
    \item For each algebra homomorphism \(\varphi: R \to S\), $\tilde{\mathcal{A}}(\varphi)=\mathcal{A}(\varphi) \times \mathcal{A}(\varphi)$.
\end{itemize}
Then $\tilde{\mathcal{A}}$ is an affine group and denoted by  \((\tilde{\mathcal{A}}, m_*, e_*, \operatorname{inv}_*)\).
Define  a binary operation \(B: \tilde{\mathcal{A}} \times\tilde{\mathcal{A}} \to \tilde{\mathcal{A}}\) by
\[
B_R(x, y) = \big( x^{\circ(-1)} \circ_R y,\; 1_R \big).
\]
Then \((\tilde{\mathcal{A}}, m_*, e_*, \operatorname{inv}_*, B)\) is an affine RB group.
\end{lemma}

\begin{proof}
First we show that \(\tilde{\mathcal{A}}\) is an affine group. It is easy to check that \(\tilde{\mathcal{A}}\) is a functor, and \(\tilde{\mathcal{A}}\) is represented by \(\mathcal{O}(\mathcal{A}) \otimes \mathcal{O}(\mathcal{A})\): for every \(R \in \mathsf{Alg}_k\),
\[
\tilde{\mathcal{A}}(R) = \mathcal{A}(R) \times \mathcal{A}(R) \cong \operatorname{Hom}(\mathcal{O}(\mathcal{A}), R) \times \operatorname{Hom}(\mathcal{O}(\mathcal{A}), R) \cong \operatorname{Hom}(\mathcal{O}(\mathcal{A}) \otimes \mathcal{O}(\mathcal{A}), R).
\]
Hence \(\tilde{\mathcal{A}}\) is a representable functor from \(\mathsf{Alg}_k\) to \(\mathsf{Set}\). By Remark~\ref{lamda from skewbrace remark}, each \(\lambda_x\) is an automorphism of \(\bigl(\mathcal{A}(R), \cdot_R\bigr)\), and
\[
\lambda: \bigl(\mathcal{A}(R), \circ_R\bigr) \longrightarrow \operatorname{Aut}\bigl(\mathcal{A}(R), \cdot_R\bigr)
\]
is a group homomorphism.  This allows us to define a group structure on each \(\tilde{\mathcal{A}}(R)\) via the semidirect product constructed using \(\lambda\):
\[
(x,y) *_R (z,t) = \bigl( x \circ_R z,\; y \cdot_R \lambda_x(t) \bigr), \qquad \forall (x,y), (z,t) \in \tilde{\mathcal{A}}(R).
\]
Furthermore, for every homomorphism \(\varphi : R \to S\), since \(\mathcal{A}(\varphi)\) is a skew brace homomorphism, we have 
\[
\begin{aligned}
	\tilde{\mathcal{A}}(\varphi)\bigl((x,y) *_R (z,t)\bigr)
	&= \Bigl( \mathcal{A}(\varphi)(x \circ_R z),\; \mathcal{A}(\varphi)\bigl(y \cdot_R \lambda_x(t)\bigr) \Bigr) \\
	&= \Bigl( \mathcal{A}(\varphi)(x) \circ_S \mathcal{A}(\varphi)(z),\; 
	\mathcal{A}(\varphi)(y) \cdot_S \lambda_{\mathcal{A}(\varphi)(x)}\bigl(\mathcal{A}(\varphi)(t)\bigr) \Bigr) \\
	&= \tilde{\mathcal{A}}(\varphi)(x,y) \;*_S\; \tilde{\mathcal{A}}(\varphi)(z,t),
\end{aligned}
\]
which implies that  the map \(\tilde{\mathcal{A}}(\varphi) = \mathcal{A}(\varphi) \times \mathcal{A}(\varphi)\) is a group homomorphism from \((\tilde{\mathcal{A}}(R), *_R)\) to \((\tilde{\mathcal{A}}(S), *_S)\). Thus \(\tilde{\mathcal{A}}\) is a group-valued functor whose underlying set-valued functor is representable. By Proposition \ref{pro:equivalent affine groups}, \(\tilde{\mathcal{A}}\) is an affine group. 

Next we prove that \((\tilde{\mathcal{A}}, m_*, e_*, \operatorname{inv}_*, B)\) is an affine RB group. As shown in the proof of Theorem~\ref{embed skew brace into RBG Thm}, the group
\[
\big(\tilde{\mathcal{A}}(R),*_R \big)=\big(\mathcal{A}(R), \circ_R\big) \ltimes \big(\mathcal{A}(R), \cdot_R\big)
\]
admits an exact factorization
\[
\tilde{\mathcal{A}}(R) = H_R *_R L_R, \quad H_R \cap L_R = \{(1_R, 1_R)\},
\]
with \(H_R = \{ (x, x) \mid x \in \mathcal{A}(R) \}\) and \(L_R = \{ (x, 1_R) \mid x \in \mathcal{A}(R) \}\). The unique decomposition is given by
\[
(x, y) = (y, y) *_R (y^{\circ(-1)} \circ_R x, 1_R).
\]
Thus, by Example \ref{prelim examples} (3), \(B_R(x, y) = (y^{\circ(-1)} \circ_R x, 1_R)^{*-1} = (x^{\circ(-1)} \circ_R y, 1_R)\) is a well-defined RB operator on the group $\big(\tilde{\mathcal{A}}(R),*_R \big)$. Therefore each \(\big(\tilde{\mathcal{A}}(R), *_R, B_R\big)\) is an RB group. 
Furthermore, the naturality of \(B\) follows from a direct computation: for any \(\varphi: R \to S\) and \((x, y) \in \tilde{\mathcal{A}}(R)\), we have
\[
\begin{aligned}
\tilde{\mathcal{A}}(\varphi)\big(B_R(x, y)\big) 
&= \tilde{\mathcal{A}}(\varphi)\big(x^{\circ(-1)} \circ_R y, 1_R \big) \\
&= \big( \mathcal{A}(\varphi)(x^{\circ(-1)} \circ_R y),\; 1_S \big) \\
&= \big( \mathcal{A}(\varphi)(x)^{\circ(-1)} \circ_S \mathcal{A}(\varphi)(y),\; 1_S \big) \\
&= B_S\big( \tilde{\mathcal{A}}(\varphi)(x, y) \big).
\end{aligned}
\]
Hence \((\tilde{\mathcal{A}}, m_*, e_*, \operatorname{inv}_*, B)\) is an affine RB group.
\end{proof}

\begin{theorem}
\label{thm:affine-skew-brace-embedding}
Every affine skew brace can be embedded into an affine RB group.
\end{theorem}

\begin{proof}
Let $(\mathcal{A},m_\bullet, m_\circ)$ be an affine skew brace, and let \((\tilde{\mathcal{A}}, m_*, e_*, \operatorname{inv}_*, B)\) be the affine RB group defined as in Lemma~\ref{Lem affine skeW brace give affine RB group}.
Define \(\Psi: \mathcal{A} \to \tilde{\mathcal{A}}\) by
\[
\Psi_R(g) = (1_R, g).
\]
Then we show \(\Psi\) is an embedding of affine skew braces 
\[
(\mathcal{A},m_\bullet, m_\circ) \to (\tilde{\mathcal{A}}, m_*, m_*^B),
\]
where the affine skew brace \((\tilde{\mathcal{A}}, m_*, m_*^B)\) is constructed by the affine RB group \((\tilde{\mathcal{A}}, m_*, e_*, \operatorname{inv}_*, B)\) given in Proposition \ref{affine RB to affine skew brace prop}, i.e., 
\[
(m_*^B)_R (x,y)=x *^B_R y=x *_R B_R(x) *_R y *_R B_R(x)^{*-1}.
\]

For any \(\varphi: R \to S\) and \(g \in \mathcal{A}(R)\),
\[
\tilde{\mathcal{A}}(\varphi)\big(\Psi_R(g)\big) = \tilde{\mathcal{A}}(\varphi)(1_R, g) = \big(1_S, \mathcal{A}(\varphi)(g)\big) = \Psi_S\big(\mathcal{A}(\varphi)(g)\big).
\]
Furthermore, if \(\Psi_R(g_1) = \Psi_R(g_2)\), then \((1_R, g_1) = (1_R, g_2)\), which implies \(g_1 = g_2\). Hence $\Psi$ is a natural transformation and $\Phi_R$ is injective for any $k$-algebra $R$.

It remains to prove that \(\Psi: (\mathcal{A},m_\bullet, m_\circ) \to (\tilde{\mathcal{A}}, m_*, m_*^B) \) is a homomorphism of affine skew braces. 
First, we have
\[
\Psi_R(x \cdot_R y) = (1_R, x \cdot_R y) = (1_R, x) *_R (1_R, y) = \Psi_R(x) *_R \Psi_R(y).
\]
Moreover, by $B_R(1_R, x) = (x, 1_R)$ and $B_R(1_R,x)^{*-1}=(x, 1_R)^{* -1} = (x^{\circ-1}, 1_R)$, we have
\[
\begin{aligned}
	\Psi_R(x) *_R^B \Psi_R(y)
	&= (1_R, x) *_R B_R(1_R, x) *_R (1_R, y) *_R B_R(1_R, x)^{*-1} \\
	&= (1_R, x) *_R (x, 1_R) *_R (1_R, y) *_R (x^{\circ-1}, 1_R) \\
	&= (x, x) *_R (1_R, y) *_R (x^{\circ -1}, 1_R) \\
	&= \big(x, x \cdot_R \lambda_x(y)\big) *_R (x^{\circ -1}, 1_R) \\
	&= (x, x \circ_R y) *_R (x^{\circ -1}, 1_R) \\
	&= (1_R, x \circ_R y) = \Psi_R(x \circ_R y).
\end{aligned}
\]
Hence  \(\Psi_R\) is a homomorphism of skew braces for any $k$-algebra $R$. Consequently, \(\Psi\) is an embedding of affine skew braces. 
\end{proof}

\subsection{Co-Rota--Baxter Hopf algebras and Hopf co-braces}

We observe that the way an affine RB group gives rise to an affine skew brace naturally yields the construction of a Hopf co-brace from a co-RB Hopf algebra. Conversely, the embedding of affine skew braces into affine RB groups leads to a new relation.

The proof of Proposition~\ref{coRBHopf give rise to Hopf cobrace} below, in contrast to earlier computational treatments, employs a categorical viewpoint and avoids lengthy computations.

\begin{proposition}{\rm (\cite{ZZL24})}
\label{coRBHopf give rise to Hopf cobrace}
    Let \((H, \mu, \eta, \Delta, \varepsilon, S, \mathcal{B})\) be a co-RB Hopf algebra and let
\[
H_{\mathcal{B}} = (H, \mu, \eta, \Delta_\mathcal{B}, \varepsilon, S_{\mathcal{B}})
\]
be its descendant Hopf algebra, where \(\Delta_B\) and \(S_B\) are defined as in Definition \ref{def descendant hopf alg}. Then the triple $(H, \Delta, \Delta_{\mathcal{B}})$ is a commutative Hopf co-brace.
\end{proposition}
\begin{proof}
A co-RB Hopf algebra \((H, \mu, \eta, \Delta, \varepsilon, S, \mathcal{B})\) together with its descendant Hopf algebra \((H, \mu, \eta, \Delta_{\mathcal{B}}, \varepsilon, S_{\mathcal{B}})\) corresponds to an affine RB group \((G, m, e, \operatorname{inv}, B)\) together with its descendant affine group \((G, m^B, e, \operatorname{inv}^B)\). Then \((G, m, m^B)\) is an affine skew brace by Proposition~\ref{affine RB to affine skew brace prop}. Theorem~\ref{affine skewbrace Hopf co-brace} then shows that \((H, \Delta, \Delta_{\mathcal{B}})\) is a commutative Hopf co-brace.
\end{proof}

\begin{proposition}
\label{new relation from Hopfcobrace to coRB}
For every commutative Hopf co-brace \((H, \Delta_\bullet, \Delta_\circ)\), there exists a commutative Hopf co-brace \((H \otimes H, \widetilde{\Delta}_\bullet, \widetilde{\Delta}_\circ)\), arising from a co-RB Hopf algebra, together with a homomorphism of Hopf co-braces
\[
(H \otimes H, \widetilde{\Delta}_\bullet, \widetilde{\Delta}_\circ) \longrightarrow (H, \Delta_\bullet, \Delta_\circ), \qquad x \otimes y \longmapsto \varepsilon_H(x) y.
\]
\end{proposition}

\begin{proof}
For any commutative Hopf co-brace \((H, \Delta_\bullet, \Delta_\circ)\), there exists an affine skew brace \((\mathcal{A}, m_\bullet, m_\circ)\) given in Theorem~\ref{affine skewbrace Hopf co-brace}. 

Theorem~\ref{thm:affine-skew-brace-embedding} provides a homomorphism of affine skew braces 
\[\Psi:(\mathcal{A}, m_\bullet, m_\circ) \to (\tilde{\mathcal{A}},m_*, m_*^B),
\]
which can be regarded as an embedding of affine skew brace \((\mathcal{A}, m_\bullet, m_\circ)\) into  the affine RB group \((\tilde{\mathcal{A}},m_*, B)\). Applying Theorem~\ref{affineRBG RBHopf THM} to \((\tilde{\mathcal{A}},m_*, B)\) yields a co-RB Hopf algebra 
\[(H\otimes H,\mu_{H\otimes H},\eta_{H\otimes H},\Delta_*, \varepsilon_*, S_*,\mathcal{B})
\] 
and then we obtain its descendant Hopf algebra
\[(H\otimes H,\mu_{H\otimes H},\eta_{H\otimes H},{\Delta_*}_{\mathcal{B}}, \varepsilon_*, {S_*}_{\mathcal{B}}).
\]
By Proposition \ref{coRBHopf give rise to Hopf cobrace},  $(H\otimes H, \Delta_*, {\Delta_*}_{\mathcal{B}})$ is a commutative Hopf co-brace. Moreover, by Theorem \ref{affine skewbrace Hopf co-brace}, the embedding $\Psi$
corresponds to a homomorphism of Hopf co-braces
\[
\psi: (H\otimes H, \Delta_*, {\Delta_*}_{\mathcal{B}}) \longrightarrow (H, \Delta_\bullet, \Delta_\circ) ,\quad x\otimes y \mapsto \varepsilon_H(x)y. \qedhere
\]
\end{proof}


\section{The Yang--Baxter equation in the setting of affine schemes}\label{sec:YB}
Recall that a set-theoretical solution to the Yang–Baxter equation is a pair \((X, r)\), where \(X\) is a set and
\[
r : X \times X \longrightarrow X \times X, \qquad r(x,y) = (\sigma_x(y), \tau_y(x)), \quad x,y \in X,
\]
is a bijective map such that
\begin{equation}
    (r \times \operatorname{id})(\operatorname{id} \times r)(r \times \operatorname{id}) = (\operatorname{id} \times r)(r \times \operatorname{id})(\operatorname{id} \times r).
\end{equation}
A solution \((X, r)\) is said to be \textbf{non-degenerate} if the maps \(\sigma_x\) and \(\tau_x\) are bijective for each \(x \in X\), and \((X, r)\) is said to be \textbf{involutive} if \(r^2 = \operatorname{id}_{X \times X}\).

Recall \cite[Theorem 3.1]{GV17} that every skew left brace $(A, \cdot, \circ)$ gives rise to a non-degenerate set-theoretical solution via the formula
\begin{equation}
    r(a,b) = (\lambda_a(b), \lambda_{\lambda_a(b)}^{-1}((a \circ b)^{-1} a (a \circ b))),
\end{equation}
where \( \lambda_a(b) = a^{-1} \cdot (a \circ b) \).  
Moreover, the solution is involutive if and only if the first group \( (A, \cdot) \) is abelian, i.e., when \( A \) is a brace.

We propose a generalisation of the set-theoretical solution to the setting of affine schemes: when $X$ is a representable functor from $\mathsf{Alg}_k$ to $\mathsf{Set}$, we require that $r: X \times X \to X \times X$ be a natural transformation. To make this precise, we give the following formal definition.

\begin{definition}
\label{def:affine-scheme-solution}
An \textbf{affine scheme solution} to the Yang--Baxter equation is a pair \((X, r)\) consisting of:
\begin{itemize}
    \item[\rm(1)] a representable functor \(X: \mathsf{Alg}_k \to \mathsf{Set}\) (i.e., an affine scheme);
    \item[\rm(2)] a natural transformation \(r: X \times X \to X \times X\)
\end{itemize}
such that, for any $R \in \mathsf{Alg}_k$, the map
\[
r_R: X(R) \times X(R) \longrightarrow X(R) \times X(R),  \quad r_R(x,y) = \big(\sigma_{R,x}(y), \tau_{R,y}(x)\big), \quad x,y \in X(R),
\]
is a set-theoretical solution to the Yang--Baxter equation.
\end{definition}


\begin{definition}
An affine scheme solution \((X, r)\) is called \textbf{non-degenerate} (resp. \textbf{involutive}) if for every \(R\), the maps \(\sigma_{R,x}, \tau_{R,x}: X(R) \to X(R)\) are bijections (resp. \(r \circ r = \operatorname{id}_{X \times X}\)).
\end{definition}

\begin{proposition}
\label{prop:affine-skew-brace-to-solution}
Let \((\mathcal{A}, m_\bullet, m_\circ)\) be an affine skew brace. For each \(R \in \mathsf{Alg}_k\) and \(a,b \in \mathcal{A}(R)\), define
\[
r_R(a,b) = \big(\lambda_a(b), \; \lambda_{\lambda_a(b)}^{-1}((a \circ_R b)^{-1} \cdot_R a \cdot_R (a \circ_R b))\big),
\]
where \(\lambda_a(b) = a^{-1} \cdot_R (a \circ_R b)\).  
Then the family \(\{r_R\}\) defines a natural transformation \(r: \mathcal{A} \times \mathcal{A} \to \mathcal{A} \times \mathcal{A}\).  
Consequently, \((\mathcal{A}, r)\) is a non-degenerate affine scheme solution to the Yang--Baxter equation. In particular,    \((\mathcal{A}, r)\) is involutive if and only if   \(\mathcal{A}\) is an affine brace.
\end{proposition}

\begin{proof}
Given an affine skew brace \((\mathcal{A}, m_\bullet, m_\circ)\). For each \(R\), the map \(r_R\) defines a non-degenerate set-theoretical solution to the Yang--Baxter equation by \cite{GV17}. And it is involutive exactly when $m_\bullet$ is abelian.

It remains to verify that \(r\) is a natural transformation. By Remark~\ref{lamda from skewbrace remark}, 
\[
\lambda: \bigl(\mathcal{A}(R), \circ_R\bigr) \longrightarrow \operatorname{Aut}\bigl(\mathcal{A}(R), \cdot_R\bigr)
\]
is a group homomorphism, which implies 
\[
\lambda_a^{-1}(b) = \lambda_{a^{\circ(-1)}}(b) = (a^{\circ(-1)})^{-1} \cdot_R (a^{\circ(-1)} \circ_R b).
\]
Thus the expression defining \(r_R(a,b)\) involves only the operations \(\cdot_R, \circ_R\) and their inverses. 
Let \(\varphi: R \to S\) be a morphism in \(\mathsf{Alg}_k\). Then $\mathcal{A}(\varphi): \mathcal{A}(R) \to \mathcal{A}(S)$ is a homomorphism of skew braces, which preserves all these operations. Consequently,
\[
(\mathcal{A}(\varphi) \times \mathcal{A}(\varphi)) \circ r_R = r_S \circ (\mathcal{A}(\varphi) \times \mathcal{A}(\varphi)),
\]
which implies that \(r\) is a natural transformation.
\end{proof}

By the relation between affine RB groups and affine skew braces, we have

\begin{corollary}
Let $(G,m,B)$ be an affine RB group. For each $R \in \mathsf{Alg}_k$ and $a,b \in G(R)$, define
\[
r_R(a,b) = \big(\lambda_a(b), \; \lambda_{\lambda_a(b)}^{-1}\big(\lambda_a(b)^{-1} \cdot_R a \cdot_R \lambda_a(b) \big)\big),
\]
where \(\lambda_a(b) = B_R(a) \cdot_R b \cdot_R B_R(a)\). 
Then \((G, r)\) is a non-degenerate affine scheme solution to the Yang--Baxter equation.
\end{corollary}
\begin{proof}
 First, by Proposition~\ref{affine RB to affine skew brace prop}, $(G,m,m^B)$ is an affine skew brace. Applying Proposition~\ref{prop:affine-skew-brace-to-solution} to it then yields the corollary.
\end{proof}

\begin{remark}
In light of Proposition~\ref{prop:affine-skew-brace-to-solution}, affine skew braces yield affine scheme solutions of the Yang–Baxter equation. Affine RB groups do so as well, via their relation to affine skew braces. Similarly, co-RB Hopf algebras and Hopf co-braces also admit such solutions, though we do not spell out the details here.
\end{remark}

$\mathbf{Acknowledgements.}$ We are deeply grateful to Andrey Lazarev, Xiaomeng Xu, and Honglei Lang for their constructive comments and useful discussions. This research is supported by NSFC (W2412041,12371029) and the National Key Research and Development Program of China (2021YFA1002000).

\end{document}